\documentclass[a4paper,10pt]{amsart}
\usepackage{amsmath,amsthm,amssymb,amsfonts,enumerate,color}

\newcommand{\R}{\mathbb{R}}

\newcommand{\Z}{\mathbb{Z}}
\newcommand{\N}{\mathbb{N}}

\newcommand{\PV}{\operatorname{P.V.}}
\newcommand{\Dom}{\operatorname{Dom}}
\newcommand{\Id}{\operatorname{I}}

\newcommand{\dive}{\operatorname{div}}

\newcommand{\tr}{\operatorname{trace}}

\newcommand{\G}{\mathcal{G}}

\newtheorem{thm}{Theorem}[section]

\newtheorem{cor}[thm]{Corollary}
\newtheorem{lem}[thm]{Lemma}

\theoremstyle{definition}
\newtheorem{defn}[thm]{Definition}
\newtheorem{rem}[thm]{Remark}

\numberwithin{equation}{section}

\allowdisplaybreaks

\author[P. R. Stinga]{Pablo Ra\'ul Stinga}

\address{Department of Mathematics\\
Iowa State University\\
396 Carver Hall, Ames\\
IA 50011, USA}
\email{stinga@iastate.edu}

\author[J. L. Torrea]{Jos\'e L. Torrea}
\address{Departamento de Matem\'aticas\\
         Universidad Aut\'onoma de Madrid\\
         28049 Madrid, Spain}
\email{joseluis.torrea@uam.es}

\thanks{The authors were supported by grant MTM2015-66157-C2-1-P (MINECO/FEDER)
from Government of Spain.}

\keywords{Nonlocal space-time equation, master equation, fractional heat operator, parabolic language of semigroups,
extension problem, Harnack inequality, Almgren frequency formula, H\"older regularity}

\subjclass[2010]{Primary: 35R09, 35R11,58J35. Secondary: 26A33, 35B65, 47G20}

\begin{document}

\title[Fractional nonlocal parabolic equations]{Regularity theory and extension problem \\
for fractional nonlocal parabolic equations \\ and the master equation}

\begin{abstract}
We develop the regularity theory for solutions to space-time nonlocal equations
driven by fractional powers of the heat operator
$$(\partial_t-\Delta)^su(t,x)=f(t,x),\quad\hbox{for}~0<s<1.$$
This nonlocal equation of order $s$ in time and $2s$ in space arises
in Nonlinear Elasticity, Semipermeable Membranes, Continuous Time Random Walks
 {and Mathematical Biology}.
It plays for space-time nonlocal equations like the generalized
master equation the same role as the fractional
Laplacian for nonlocal in space equations.
We obtain a pointwise integro-differential formula for $(\partial_t-\Delta)^su(t,x)$
and parabolic maximum principles.
A novel extension problem to characterize this nonlocal equation 
with a local degenerate parabolic equation is proved.
We show parabolic interior and boundary Harnack inequalities,
 {and an Almgrem-type monotonicity formula}.
H\"older and Schauder estimates for the space-time Poisson problem
are deduced using a  new characterization of parabolic H\"older spaces.
Our methods involve the \textit{parabolic language of semigroups}
and the Cauchy Integral Theorem,
which are original to define the fractional powers of $\partial_t-\Delta$.
Though we mainly focus in the equation $(\partial_t-\Delta)^su=f$,
applications of our ideas to variable coefficients, discrete Laplacians
and Riemannian manifolds are stressed out.
\end{abstract}

\maketitle

\section{Introduction}

We develop the study of regularity and fine properties of solutions $u=u(t,x)$
to space-time nonlocal equations driven by the fractional powers of the heat operator  {$H=\partial_t-\Delta$},
\begin{equation}\label{ecuacion}
H^su\equiv(\partial_t-\Delta)^su=f,\quad\hbox{for}~0<s<1,~(t,x)\in\R^{n+1}.
\end{equation}

The space-time nonlocal problem $H^su=f$ arises from several applications in Nonlinear Elasticity,
Semipermeable Membranes, Continuous Time Random Walks (CTRW)
and Mathematical Biology,
just to mention a few. Indeed, it is not difficult to verify that the flat parabolic Signorini problem
(see the book by Duvaut and Lions \cite[Section~2.2.1]{Duvaut-Lions},
also \cite{Danielli-Garofalo-Petrosyan-To}) is equivalent to the obstacle problem for $(\partial_t-\Delta)^{1/2}$.
 {Similarly, the diffusion model for biological
invasions introduced in \cite{Berestycki} is equivalent to a local-nonlocal system coupling
a classical heat equation with an equation for $(\partial_t-\Delta)^{1/2}$.
The obstacle problem for $(\partial_t-\Delta)^s$, $0<s<1$, has been recently considered by
Athanasopoulos, Caffarelli and Milakis in \cite{ACM2}.}
Moreover, as we show in \eqref{eq:pointwise formula}, the equation \eqref{ecuacion} is a master equation.
The master equation is fundamental in the theory of CTRW,
see Metzler--Klafter \cite{Metzler-Klafter} and Caffarelli--Silvestre \cite{Caffarelli-Silvestre Master}.
It is worth noting that in \eqref{ecuacion} the random jumps are \textit{coupled} with the random waiting times.
This is in contrast with equations like $\partial_tu+(-\Delta)^su=f$ or
$D_t^\alpha u+(-\Delta)^su=f$, where jumps are \textit{independent} of the waiting times.
We finally mention that \eqref{ecuacion} is stated as equation (4.4) in \cite{Baeumer}.

The equation in \eqref{ecuacion} can be regarded as one of the most basic space-time nonlocal master equations
of order $s$ in time and $2s$ in space.
The fractional powers of the heat operator play
for these types of equations the parallel role as that of the fractional powers of the Laplacian in the general
theory of integro-differential equations.

In this paper the analysis of fine regularity properties is carried out.
Using the \textit{parabolic language of semigroups} we find the pointwise formula
for $H^su$ and its inverse $H^{-s}f$ (fundamental solution) and prove maximum principles.
We obtain a novel extension problem that characterizes this nonlocal operator through
a local degenerate evolution PDE.
The extension equation turns out to be a parabolic degenerate equation in one more variable.
Interior and boundary Harnack inequalities of parabolic type, and an Almgren-type monotonicity
formula, are deduced. We also prove fractional parabolic H\"older and Schauder estimates.

Apart from their immediate interest in applications (see 
\cite{ACM2}, \cite[eq.~(4.4)]{Baeumer}, \cite{Berestycki},
\cite{Caffarelli-Silvestre Master}, \cite{Danielli-Garofalo-Petrosyan-To}, \cite{Duvaut-Lions} and \cite{Metzler-Klafter}),
our novel ideas and results open the way to consider
anisotropic equations of the form $(\partial_t-L)^su=f$ on Riemannian manifolds, where $L$ is an elliptic
operator with bounded measurable coefficients that may depend on $t$ and $x$.
Observe that it is not clear at all how to define these nonlocal parabolic 
operators or how to obtain pointwise expressions for them.
We present here novel and crucial ideas that allow us to get the
new interesting results mentioned above.
In particular, we use tools such as the Cauchy Integral Theorem and analytic continuation
in combination with an original technique that we call \textit{parabolic language of semigroups}
and local parabolic PDE techniques.
Moreover, our method gives a  novel and very useful characterization of parabolic H\"older spaces.  {Some antecedents of the latter in the elliptic case can be found in the book by Stein \cite{Stein}.}

Let us begin by precisely defining the fractional powers of the heat operator.
For $0<s<1$, the fractional heat operator $H^su(t,x)$
of a function $u=u(t,x):\R^{n+1}\to\R$, $n\geq1$, is given as
$$\widehat{H^su}(\rho,\xi)=(i\rho+|\xi|^2)^s\widehat{u}(\rho,\xi),$$
for $\rho\in\R$ and $\xi\in\R^n$. Observe that
the Fourier multiplier of the heat operator
$H$ is the complex number $i\rho+|\xi|^2$.
Hence a first goal is to give a correct meaning to the $s$-power of such a complex number.
This can be solved by performing the analytic continuation of the function $z\mapsto z^s$, to $\Re(z)>0$.
Even though this consideration is correct as a definition, we still have the problem of finding a
pointwise formula for $H^su$ at any space-time point $(t,x)$. One could try to compute the inverse
Fourier transform in the definition above.
This methodology, though plausible, would require quite involved and nontrivial calculations.
But still, such approach will be completely useless to handle other equations
like $(\partial_t-L)^su=f$ when $L$ is, for example, non translation invariant.
We overcome these difficulties by introducing a novel original idea: the parabolic language of semigroups.
One of the main advantages of this new method, among many others, is that it allows us to 
avoid the computation of the inverse Fourier transform.  {We shall write $A\sim B$ 
as a shorthand notation to denote the existence of two constants
$c_1,c_2>0$ such that $c_1A\le B\le c_2A$.}

\begin{thm}[Pointwise formula]\label{thm:puntual}
Let $u$ be a function
in the parabolic H\"older space $C^{s+\varepsilon,2s+\varepsilon}_{t,x}(\R^{n+1})$,
for some $\varepsilon>0$ (see Section \ref{Section:puntual} for definitions).
Then, for every $(t,x)\in\R^{n+1}$,
\begin{equation}\label{eq:pointwise formula}
H^su(t,x)=\int_0^\infty\int_{\R^n}\big(u(t,x)-u(t-\tau,x-z)\big)K_s(\tau,z)\,dz\,d\tau,
\end{equation}
where
$$K_s(\tau,z)=\frac{1}{(4\pi)^{n/2}|\Gamma(-s)|}\cdot\frac{e^{-|z|^2/(4\tau)}}{\tau^{n/2+1+s}},$$
for $\tau>0$, $z\in\R^n$. Here $\Gamma$ denotes the Gamma function.
Moreover, for some $0<\lambda\leq\Lambda$,
\begin{enumerate}[$(1)$]
\item $\displaystyle K_s(\tau,z)\geq\frac{\lambda}{|z|^{n+2+2s}}$, when $\tau\sim|z|^2$,
\item $\displaystyle K_s(\tau,z)\leq\frac{\Lambda}{|z|^{n+2+2s}+\tau^{(n+2+2s)/2}}$, for every $z\neq0$.
\end{enumerate}
\end{thm}

Our pointwise formula in \eqref{eq:pointwise formula} shows that
$H^su(t,x)=f(t,x)$ is indeed a
master equation like the ones appearing in \cite{Caffarelli-Silvestre Master, Metzler-Klafter}.
It also provides the explicit formula for the equation in \cite[eq.~(4.4)]{Baeumer}.
In particular, estimates (1) and (2) for the kernel $K_s(\tau,z)$ above
are the same that Caffarelli and Silvestre assume in \cite{Caffarelli-Silvestre Master}
for $\sigma=2s$ and $\beta=2$. Therefore, we regard the equation $H^su=f$ as 
the most basic space-time nonlocal master equation.

\begin{rem}[Dirichlet problem -- Boundary condition]
Notice from \eqref{eq:pointwise formula} that in order to compute $H^su(t,x)$ we need to know
the values of $u(\tau,z)$ for every $\tau\leq t$ (the past) and every $z\in\R^n$.
Then the natural Dirichlet problem for $H^s$ takes the form
\begin{equation}\label{eq:Dirichlet problem}
\begin{cases}
H^su=f,&\hbox{in}~(T_0,T_1]\times\Omega,\\
u=g,&\hbox{in}~\big[(-\infty,T_0]\times\R^n\big]\cup\big[(T_0,T_1)\times(\R^n\setminus\Omega)\big],
\end{cases}
\end{equation}
where $\Omega$ is a bounded domain in $\R^n$, $T_0<T_1$ are real numbers and $f(t,x)$ and
$g(t,x)$ are prescribed functions in their respective domains.
\end{rem}
 
\begin{rem}[Scaling]
It follows from \eqref{eq:pointwise formula} that if $u_\lambda(t,x)=u(\lambda^2t,\lambda x)$, $\lambda>0$, then
we have the fractional parabolic scaling
\begin{equation}\label{eq:scaling}
(H^su_\lambda)(t,x)=\lambda^{2s}(H^su)(\lambda^2t,\lambda x).
\end{equation}
\end{rem}

If a function $u$ does not depend on $t$ then $Hu=-\Delta u$, while if it does not depend on $x$ then
$Hu=\partial_tu$. In a similar way we obtain

\begin{cor}[Separated variables cases]\label{cor:separate}
Let $u\in C^{s+\varepsilon,2s+\varepsilon}_{t,x}(\R^{n+1})$,
for some $\varepsilon>0$.
If $u$ does not depend on $t$ then
$$H^su(t,x)=(-\Delta)^su(x),\quad x\in\R^n,$$
the fractional Laplacian on $\R^n$. If $u$ does not depend on $x$ then
$$H^su(t,x)=(\partial_t)^su(t)=\frac{1}{|\Gamma(-s)|}\int_{-\infty}^t\frac{u(t)-u(\tau)}{(t-\tau)^{1+s}}\,d\tau,
\quad t\in\R,$$
the Marchaud fractional derivative of order $s$,  {see \cite{Marchaud}}.
\end{cor}

It is worth noticing that the appearance of the Marchaud fractional derivative
\textit{from the past} is related to the fact that the heat equation is not time reversible,
see also the detailed analysis in \cite{bernardis}.

The parabolic nonlocal maximum principle takes the following form.

\begin{thm}[Maximum principle]\label{thm:maximum}
Let $u\in C^{s+\varepsilon,2s+\varepsilon}_{t,x}(\R^{n+1})$,
for some $\varepsilon>0$. Suppose that
\begin{enumerate}[$(1)$]
\item $u(t_0,x_0)=0$ for some $(t_0,x_0)\in\R^{n+1}$, and
\item $u(t,x) \ge 0$ for $t\le t_0$, $x\in\mathbb{R}^n$.
\end{enumerate}
Then
$$H^su(t_0,x_0)\le 0.$$
Moreover, $H^su(t_0,x_0)=0$ if and only if $u(t,x)=0$ for $t\le t_0$ and $x\in \mathbb{R}^n$.
\end{thm}

As a consequence, we have a comparison principle and 
the uniqueness of solutions to the Dirichlet problem \eqref{eq:Dirichlet problem}.

\begin{cor}[Comparison principle -- Uniqueness]\label{cor:comparison}
Let $\Omega$ be a bounded domain of $\R^n$ and let $T_0<T_1$
be two real numbers.  {Suppose that $u,w\in C^{s+\varepsilon,2s+\varepsilon}_{t,x}(\R^{n+1})$,
for some $\varepsilon>0$.} If
$$\begin{cases}
H^s u\geq H^sw,&\hbox{in}~(T_0,T_1]\times\Omega,\\
u\geq w,&\hbox{in}~\big[(-\infty,T_0]\times\R^n\big]\cup\big[(T_0,T_1)\times(\R^n\setminus\Omega)\big],
\end{cases}$$
then $u \geq w$ in $(-\infty,T_1]\times\R^n$. In particular, uniqueness for the Dirichlet problem
\eqref{eq:Dirichlet problem} holds.
\end{cor}

We characterize the fractional heat operator, which is a \textit{nonlocal} operator,
with a \textit{local} extension problem in one dimension more.
The extension equation we obtain is 
a parabolic degenerate equation with Muckenhoupt $A_2$-weight.
The new variable $y$ is a spatial variable for the new equation.
For the case of the fractional powers of the Laplacian an \textit{elliptic} extension problem
is available \cite{Caffarelli-Silvestre CPDE, Stinga-Torrea CPDE},
see \cite{GMS} where the most general extension problem
that can be found in the literature was proved,
namely, for fractional powers of operators in Banach spaces.
However we stress that, for the fractional heat operator, our new extension problem is \textit{parabolic}.
The extension problem for the Marchaud fractional derivative has been originally obtained in \cite{bernardis},
where the extension PDE is also of parabolic type.
For the definition of $e^{-\tau H}u(t,x)$ see Section \ref{Section:puntual}.

\begin{thm}[Extension problem]\label{thm:extension}
Let $0<s<1$ and take
$u\in\Dom(H^s)=\big\{u\in L^2(\R^{n+1}):(i\rho+|\xi|^2)^s\widehat{u}(\rho,\xi)\in L^2(\R^{n+1})\big\}$.
For $(t,x)\in\R^{n+1}$ and $y>0$ define
\begin{equation}\label{eq:U}
\begin{aligned}
U(t,x,y) &= \frac{y^{2s}}{4^s\Gamma(s)}
\int_0^\infty e^{-y^2/(4\tau)}e^{-\tau H}u(t,x)\,\frac{d\tau}{\tau^{1+s}} \\
&= \frac{1}{\Gamma(s)}\int_0^\infty e^{-r}e^{-\frac{y^2}{4r}H}u(t,x)\,\frac{dr}{r^{1-s}} \\
&= \frac{1}{\Gamma(s)}\int_0^\infty
e^{-y^2/(4r)}e^{-rH}\big(H^su\big)(t,x)\,\frac{dr}{r^{1-s}}.
\end{aligned}
\end{equation}
 {Then $U(\cdot,\cdot,y)\in C^\infty((0,\infty);\Dom(H))\cap C([0,\infty);L^2(\R^{n+1}))$ solves}
\begin{equation}\label{eq:extension}
\begin{cases}
\partial_tU=\Delta U+\frac{1-2s}{y}U_y+U_{yy},&\hbox{for}~(t,x)\in\R^{n+1},~y>0,\\
\displaystyle\lim_{y\to0^+}U(t,x,y)=u(t,x),&\hbox{in}~L^2(\mathbb{R}^{n+1}),
\end{cases}
\end{equation}
and
\begin{equation}\label{eq:Neumann}
-\lim_{y\to0^+}\frac{U(t,x,y)-U(t,x,0)}{y^{2s}}=
\frac{|\Gamma(-s)|}{4^s\Gamma(s)} H^su(t,x)= -\frac1{2s} \lim_{y\to0^+}y^{1-2s}U_y(t,x,y),
\end{equation}
 {in $L^2(\mathbb{R}^{n+1})$}.
In addition, for $(t,x)\in\R^{n+1}$, $y>0$, we can also write the Poisson formula
\begin{equation}\label{eq:Poisson formula}
U(t,x,y)=\int_0^\infty\int_{\R^n} P_y^{s}(\tau,z)u(t-\tau,x-z)\,dz\,d\tau,
\end{equation}
where the fractional Poisson kernel $P_y^s(\tau,z)$ is defined for $\tau>0$, $z\in\R^n$ and $y>0$ by
\begin{equation}\label{eq:PoissonKernel}
P_y^s(\tau,z)=\frac{1}{4^{n/2+s}\pi^{n/2}\Gamma(s)}\cdot\frac{y^{2s}}{\tau^{n/2+1+s}}\,e^{-(y^2+|z|^2)/(4\tau)}.
\end{equation}
\end{thm}

 {Observe that the hypoellipticity of parabolic operators with smooth coefficients implies that
any $L^2$ solution to the extension equation in \eqref{eq:extension} is smooth
in $(t,x)\in\R^{n+1}$, $y>0$}.

\begin{thm}[Fundamental solution for the extension problem]\label{thm:fundamental}
Let $f=f(t,x)$ be a smooth function with compact support. Define 
\begin{equation}\label{eq:G}
G_s(\tau,z,y):=\frac{1}{\Gamma(s)(4\pi)^{n/2}}\cdot\frac{e^{-(|z|^2+y^2)/(4\tau)}}{\tau^{n/2+1-s}},
\end{equation}
for $(\tau,z)\in\R^{n+1}$, $y>0$. Then the function
$$\mathcal{U}(t,x,y):=\int_0^\infty\int_{\R^n}G_s(\tau,z,y)f(t-\tau,x-z)\,dz\,d\tau,$$
is a classical solution of the extension equation \eqref{eq:extension} and
$$-\lim_{y\to0^+}\frac{\mathcal{U}(t,x,y)-\mathcal{U}(t,x,0)}{y^{2s}}
=\frac{|\Gamma(-s)|}{4^s\Gamma(s)}f(t,x)=-\frac{1}{2s}\lim_{y\to0^+}y^{1-2s}\partial_y\mathcal{U}(t,x,y).$$
\end{thm}

\begin{rem}[Relation between fundamental solution and Poisson kernel for the extension problem]
The Poisson kernel in \eqref{eq:PoissonKernel} (which also appears in \cite[p.~309]{AC}) is given by
the conormal derivative of the fundamental solution \eqref{eq:G}:
$$-\frac{\Gamma(1-s)}{4^{s-1/2}\Gamma(s)}y^{1-2(1-s)}\partial_yG_{1-s}(\tau,z,y)=P_y^s(\tau,z).$$
\end{rem}

\begin{rem}[Divergence and non divergence structures]
The extension equation in \eqref{eq:extension} can be written in divergence form as
$$\partial_tU=y^{-(1-2s)}\dive_{x,y}(y^{1-2s}\nabla_{x,y}U).$$
This is a degenerate equation with Muckenhoupt $A_2$-weight given by $\omega(x,y)=y^{1-2s}$. Degenerate parabolic
equations with these types of weights have been studied by
Chiarenza--Serapioni \cite{Chiarenza-Serapioni} and Guti\'errez--Wheeden \cite{Gutierrez-Wheeden}.
We can also write the extension equation in non divergence form using
the change of variables $z=(y/(2s))^{2s}$ as
\begin{equation}\label{eq:nondivergence}
\partial_tU=\Delta U+z^{2-1/s}U_{zz},\quad z>0.
\end{equation}
In this case we get $y^{1-2s}U_y=(2s)^{1-2s}U_z$, so the Neumann boundary condition \eqref{eq:Neumann} reads
$$-U_z(t,x,0)=\frac{s^{2s}\Gamma(1-s)}{\Gamma(1+s)}H^su(t,x).$$
The same change of variables applied to \eqref{eq:PoissonKernel} and
\eqref{eq:G} allows to compute the Poisson kernel and the fundamental solution of
the non divergence form extension equation \eqref{eq:nondivergence}, respectively.
\end{rem}

Let us recall the parabolic Harnack inequality for caloric functions. There exists
a constant $c>0$ depending only on $n$ such that if $v$ is a solution to
\begin{equation}\label{eq:local}
\begin{cases}
\partial_tv-\Delta v=0,&\hbox{in}~R:=(0,1)\times B_2,\\
v\geq0,&\hbox{in}~R,
\end{cases}
\end{equation}
then
$$\sup_{R^-}v\leq c\inf_{R^+}v,$$
where $R^-=(1/4,1/2)\times B_1$, $R^+=(3/4,1)\times B_1$.
This result was discovered independently by Pini \cite{Pini} and Hadamard \cite{Hadamard}.
For the Harnack inequality for uniformly parabolic equations in divergence form 
see Moser \cite{Moser}, and for degenerate parabolic equations in divergence
form see Chiarenza--Serapioni \cite{Chiarenza-Serapioni}, also 
Guti\'errez--Wheeden \cite{Gutierrez-Wheeden} and Ishige \cite{Ishige}.

Next we obtain the parabolic interior and boundary Harnack inequalities for fractional caloric functions,
and interior H\"older and derivative estimates.
As the random jumps and waiting times are coupled, until now it was not clear at all whether
or not Harnack estimates of parabolic type would be true.

\begin{thm}[Parabolic Harnack inequality]\label{thm:Harnack}
There is a constant $c>0$ depending only on $n$ and $0<s<1$ such that
if $u=u(t,x)\in\Dom(H^s)$ is a solution to
$$\begin{cases}
H^su=0,&\hbox{in}~R=(0,1)\times B_2,\\
u\geq0,&\hbox{in}~(-\infty,1)\times\R^n,
\end{cases}$$
then
$$\sup_{R^-}u\leq c\inf_{R^+}u.$$
\end{thm}

\begin{rem}[Non negativity condition in Harnack inequality]
Observe that in the local case of the heat equation \eqref{eq:local} we just need the solution $v$
to be nonnegative in $R$. In the case of the fractional heat operator, a sufficient condition
for the Harnack inequality is $u\geq0$ everywhere in $(-\infty,1)\times\R^n$. However, 
this condition is not strictly necessary.
Indeed, in our proof (see Section \ref{Section:Harnack}) we only need $U$ in the extension problem \eqref{eq:extension}
to be nonnegative (see \eqref{eq:Poisson formula}) \textit{above} the rectangle $R$.
A parallel remark applies to the parabolic boundary Harnack inequality
contained in Theorem \ref{thm:boundary Harnack}.
\end{rem}

\begin{cor}[Interior estimates]\label{cor:interior}
 {Let $u\in\Dom(H^s)\cap L^\infty((-\infty,1)\times\R^n)$ be a solution to
\begin{equation}\label{enD}
H^su=0,\quad\hbox{in}~(0,1)\times B_2.
\end{equation}
\begin{enumerate}[$(1)$]
\item \emph{(H\"older continuity).} There exist constants $0<\alpha<1$ and 
$C>0$ depending on $n$ and $s$ such that
$$|u(t,x)-u(\tau,z)|\leq C(|t-\tau|^{1/2}+|x-z|)^\alpha\|u\|_{L^\infty((-\infty,1)\times\R^{n})},$$
for every $(t,x),(\tau,z)\in(1/2,1)\times B_1$.
\item \emph{(Derivative estimates).} For any integer $k\geq0$ and any multi-index $\beta\in\N_0^n$
there is a constant $C=C(k,|\beta|,n,s)>0$ such that
$$\sup_{(1/4,3/4)\times B_1}|\partial_t^kD^\beta_xu(t,x)|\leq C\|u\|_{L^\infty((-\infty,1)\times\R^{n})}.$$
In particular, bounded solutions $u$ to $H^su=0$ in $(0,1)\times B_2$ are $C^\infty$ in the interior.
\end{enumerate}}
\end{cor}

Next, let us recall the boundary Harnack inequality for caloric functions, see Kemper \cite{Kemper}.
Let $\Omega$ be a bounded domain of $\R^n$ and suppose that $0\in\partial\Omega$.
We assume that the boundary of $\Omega$ inside the ball $B_2$
can be described as the graph of a Lipschitz function $\psi$ on $\R^{n-1}$
with Lipschitz constant, say, 1, in the $e_n$ direction. More precisely, we assume that $\psi(0)=0$,
$$\Omega\cap B_2=\{(x',x_n)\in\R^n:x_n>\psi(x')\}\cap B_2,$$ 
and that $\partial\Omega\cap B_2=\{(x',\psi(x')):x'\in\R^{n-1}\}\cap B_2$.
 {Fix a point $(t_0,x_0)\in (-2,2)\times\Omega$ such that $t_0>1$. 
There exists a positive constant $C$ such that for every solution
$v=v(t,x)$ to}
$$\begin{cases}
\partial_tv-\Delta v=0,&\hbox{in}~(-2,2)\times\Omega,\\
v\geq0,&\hbox{in}~(-2,2)\times\Omega,
\end{cases}$$
such that
\begin{center}
$v$ vanishes continuously on $(-2,2)\times(\partial\Omega\cap B_2)$,
\end{center}
we have
$$\sup_{(-1,1)\times(\Omega\cap B_1)}v(t,x)\leq Cv(t_0,x_0).$$
 {The constant $C$ depends on $n$, $t_0-1$ and the geometry of $\Omega$.}
For the boundary Harnack inequality
for uniformly parabolic equations in divergence form see Salsa \cite{Salsa}
and for degenerate parabolic equations in divergence form see Ishige \cite{Ishige}.

\begin{thm}[Boundary Harnack inequality]\label{thm:boundary Harnack}
Consider $\Omega$ with the geometry described in the previous paragraph.
Let $u=u(t,x)\in\Dom(H^s)$ be a solution to 
$$\begin{cases}
H^su=0,&\hbox{in}~(-2,2)\times\Omega,\\
u\geq0,&\hbox{in}~(-\infty,2)\times\R^n,
\end{cases}$$
such that
\begin{center}
$u$ vanishes continuously on $(-2,2)\times((\R^n\setminus\Omega)\cap B_2)$.
\end{center}
Let $(t_0,x_0)$ be a fixed point in $(-2,2)\times\Omega$ such that $t_0>1$. Then
$$\sup_{(-1,1)\times(\Omega\cap B_1)}u(t,x)\leq Cu(t_0,x_0),$$
where  {$C>0$ depends only on $n$, $s$, $t_0-1$ and the geometry of $\Omega$}.
\end{thm}

See Remark \ref{rem:3} for Harnack estimates in the case of bounded measurable coefficients.

 {Monotonicity formulas are a very powerful tool in the study of the regularity
properties of PDEs. They have been used in a number of free boundary problems to exploit
the local properties of the equations by giving information about blowup configurations.
We present a parabolic Almgren-type
frequency formula for $H^s$. The formula is useful, for example,
to analyze the obstacle problem for $H^s$ with zero obstacle,
or to obtain unique continuation properties. We will not address these applications here.
In the case $s=1/2$ our frequency formula reduces to the
formula for the heat equation $\partial_tU-\Delta_{x,y}U=0$ established by
C.-C. Poon in \cite{Poon}, see also
\cite[Chapter~6]{Danielli-Garofalo-Petrosyan-To}.
A different monotonicity formula was recently proved by Athanasopoulos, Caffarelli
and Milakis in \cite{ACM2}. In the elliptic case, namely, for the fractional Laplacian, an Almgren-type
frequency formula was obtained by Caffarelli and Silvestre in \cite{Caffarelli-Silvestre CPDE}.
In our next result, for simplicity, we denote
$$X=(x,y)\in\R^{n+1}_+,~\dive=\dive_{x,y},~\nabla=\nabla_{x,y},~a=1-2s\in(-1,1).$$
We define the \textit{backwards} fundamental solution for the extension problem, see \eqref{eq:G}, as
$$\G_s(t,X)=\frac{1}{\Gamma(s)(4\pi)^{n/2}}\cdot\frac{e^{|X|^2/(4t)}}{(-t)^{n/2+1-s}}\chi_{t<0}.$$}

\begin{thm}[Almgren frequency  formula]\label{thm:Almgren}
 {Let $U=U(t,X)$ be a smooth function of $t\in(-1,0)$ and $X\in\R^{n+1}_+$ such that
$U(t,\cdot),\nabla U(t,\cdot)\in L^2(\R^{n+1}_+,y^adX)$, for all $t\in(-1,0)$.
Suppose that
$$y^a\partial_tU=\dive(y^a\nabla U),\quad\hbox{for}~t\in(-1,0),~X\in\R^{n+1}_+,$$
and
$$\hbox{either}\quad \lim_{y\to0^+}U(t,x,y)=0\quad\hbox{or}\quad-\lim_{y\to0^+}y^aU_y(t,x,y)=0,$$
for every $t\in(-1,0)$ and $x\in\mathbb{R}^{n}$. Let 
$$I(R)=R^2\int_{\R^{n+1}_+}y^a|\nabla U(-R^2,X)|^2\G_s(-R^2,X)\,dX,$$
and
$$H(R)=\int_{\R^{n+1}_+}y^a|U(-R^2,X)|^2\G_s(-R^2,X)\,dX,$$
for $R\in(0,1)$. Then the Almgren frequency function
$$N(R)=I(R)/H(R)$$
is monotone nondecreasing in $R\in(0,1)$. Moreover, $N(R)$ is constant if and only
if $U$ is parabolically homogeneous of degree $\kappa\in\R$,
that is, $U(\lambda^2t,\lambda X)=\lambda^\kappa U(t,X)$.}
\end{thm}

From the estimates for the kernel $K_{s}$ in Theorem \ref{thm:puntual},
the fractional parabolic scaling \eqref{eq:scaling} and Corollary \ref{cor:separate},
we see that $H^s$ is an operator of order $s$ in time and $2s$ in space.
This observation becomes more evident with our new global H\"older estimates.

\begin{thm}[H\"older estimates]\label{thm:Holder}
If $u\in C^{\alpha/2,\alpha}_{t,x}(\R^{n+1})$ for $2s<\alpha$
then $H^su\in C^{\alpha/2-s,\alpha-2s}_{t,x}(\R^{n+1})$ with the estimate
$$\|H^su\|_{C^{\alpha/2-s,\alpha-2s}_{t,x}}\leq C_{n,s}\|u\|_{C^{\alpha/2,\alpha}_{t,x}}.$$
\end{thm}

In the following result we show how to solve the global nonlocal Poisson problem
\begin{equation}\label{eq:Poisson problem}
H^su=f,\quad\hbox{in}~\R^{n+1}.
\end{equation}

\begin{thm}[Poisson problem -- Fundamental solution for $H^s$]\label{thm:solution Poisson problem}
Let $f=f(t,x)$ be an $L^2(\R^{n+1})$ function with enough decay at infinity.
Then an $L^2(\R^{n+1})$ solution $u=u(t,x)$ to \eqref{eq:Poisson problem} is given by
\begin{align*}
u(t,x) &= \frac{1}{\Gamma(s)}\int_0^\infty e^{-\tau H}f(t,x)\,\frac{d\tau}{\tau^{1-s}} \\
&=\int_0^\infty\int_{\R^n}K_{-s}(\tau,z)f(t-\tau,x-z)\,dz\,d\tau,
\end{align*}
where
$$K_{-s}(\tau,z)=\frac{1}{(4\pi)^{n/2}\Gamma(s)}\cdot\frac{e^{-|z|^2/(4\tau)}}{\tau^{n/2+1-s}},$$
for $\tau>0$, $z\in\R^n$. Moreover, for some $0<\lambda\leq\Lambda$ we have the estimates
\begin{enumerate}[$(1)$]
\item $\displaystyle K_{-s}(\tau,z)\geq\frac{\lambda}{|z|^{n+2-2s}}$, when $\tau\sim|z|^2$,
\item $\displaystyle K_{-s}(\tau,z)\leq\frac{\Lambda}{|z|^{n+2-2s}+\tau^{(n+2-2s)/2}}$, for every $z\neq0$.
\end{enumerate}
\end{thm}

\begin{rem}[Separated variables in the Poisson problem]
If in Theorem \ref{thm:solution Poisson problem} the right hand side $f$ does not depend on $t$ then
$u$ is nothing but the convolution of $f$ with the fundamental solution of the fractional Laplacian $u=(-\Delta)^{-s}f$.
When $f$ does not depend on $x$ then
$$u(t)=(\partial_t)^{-s}f(t)=\frac{1}{\Gamma(s)}\int_{-\infty}^t\frac{f(\tau)}{(t-\tau)^{1-s}}\,d\tau,$$
that is, $u$ is the Weyl fractional integral of $f$, see \cite{bernardis, Sa}.
Observe that $(\partial_t)^{-s}$  is the inverse of the Marchaud
fractional derivative appearing in Corollary \ref{cor:separate}.
In fact, a Fractional Fundamental Theorem of Calculus
on general one-sided weighted Lebesgue spaces was recently proved in \cite{bernardis}.
\end{rem}

The definitions of H\"older and Zygmund spaces are contained in Sections \ref{Section:puntual}
and \ref{Section:Holder}, respectively.

\begin{thm}[Global Schauder estimates]\label{thm:Schauder}
Let $u$ be a solution to $H^su=f$ in $\R^{n+1}$ as described in Theorem \ref{thm:solution Poisson problem}.
\begin{enumerate}[$(1)$]
	\item Let $f\in C^{\alpha/2,\alpha}_{t,x}(\R^{n+1})$ for some $0<\alpha\le1$.
		\begin{enumerate}[$(a)$]
			\item If $\alpha+2s$ is not an integer
			then $u$ is in $C^{\alpha/2+s,\alpha+2s}_{t,x}(\R^{n+1})$ with the estimate
			$$\|u\|_{C^{\alpha/2+s,\alpha+2s}_{t,x}}\leq C\big(\|f\|_{C^{\alpha/2,\alpha}_{t,x}}+\|u\|_{L^\infty(\R^{n+1})}\big).$$
			\item If $\alpha+2s$ is an integer
			then $u$ is in the parabolic H\"older--Zygmund space 
			$\Lambda^{\alpha/2+s,\alpha+2s}_{t,x}$ with the estimate
			$$\|u\|_{\Lambda^{\alpha/2+s,\alpha+2s}_{t,x}}\leq C\big(\|f\|_{C^{\alpha/2,\alpha}_{t,x}}+\|u\|_{L^\infty(\R^{n+1})}\big).$$
		\end{enumerate}
		The constants $C$ in $(a)$ and $(b)$ depend only on dimension, $\alpha$ and $s$.		
	\item Let $f\in L^\infty(\R^{n+1})$.
	\begin{enumerate}[$(i)$]
		\item If $0<s<1/2$ then $u\in C^{s,2s}_{t,x}(\R^{n+1})$ with the estimate
		$$\|u\|_{C^{s,2s}_{t,x}}\leq C\big(\|f\|_{L^\infty(\R^{n+1})}+\|u\|_{L^\infty(\R^{n+1})}\big).$$
		\item If $s=1/2$ then $u$ is in the parabolic H\"older-Zygmund space $\Lambda^{1/2,1}_{t,x}$ with the estimate
		$$\|u\|_{\Lambda^{1/2,1}_{t,x}}\leq C\big(\|f\|_{L^\infty(\R^{n+1})}+\|u\|_{L^\infty(\R^{n+1})}\big).$$
		\item If $1/2<s<1$ then $u\in C^{s,1+(2s-1)}_{t,x}(\R^{n+1})$ with the estimate
		$$\|u\|_{C^{s,1+(2s-1)}_{t,x}}\leq C\big(\|f\|_{L^\infty(\R^{n+1})}+\|u\|_{L^\infty(\R^{n+1})}\big).$$
	\end{enumerate}
	The constants $C$ in $(i)$, $(ii)$ and $(iii)$ depend only on dimension and $s$.
\end{enumerate} 
\end{thm}

 {One of our main contributions here is not only the estimates above,
but the new characterization of the parabolic H\"older space (defined pointwise in Krylov \cite{Krylov})
and the definition and  characterization of parabolic Zygmund
spaces obtained in Theorems \ref{caract} and \ref{prop:2}
in Section \ref{Section:Holder}. These results are of independent interest.
Parallel results can be found in the elliptic case, see \cite[Chapter V]{Stein}.}  
These descriptions permit us to use our formulas \eqref{1} and \eqref{2}
to derive the H\"older and Schauder estimates.
This is in contrast with other regularity methods that use the pointwise formulas \cite{Krylov,Silvestre Thesis}
or Fourier multipliers techniques \cite{Madych-Riviere}. On top of that, our method is widely applicable.
Observe that if $f$ does not depend on $t$ then we recover the 
Schauder estimates for the fractional Laplacian \cite{Caffarelli-Stinga, Silvestre Thesis},
while if $f$ does not depend on $x$ then we get Schauder estimates
for the Marchaud fractional derivative, see also \cite{Sa}.
We mention that some related fractional powers associated to space-time derivative operators
can be found in the exhaustive monograph \cite{Sa}.
For the local case $s=1$, parabolic Schauder estimates can be found, for example,
in Krylov \cite[Chapter~8]{Krylov}.

Our last main result contains the local Schauder estimates.

\begin{thm}[Local Schauder estimates]\label{thm:interiorSchauder}
 {Let $u\in\Dom(H^s)\cap L^\infty((-\infty,1)\times\R^n)$ be a solution to
$$H^su=f,\quad\hbox{in}~(0,1)\times B_2.$$
\begin{enumerate}[$(1)$]
	\item Let $f\in C^{\alpha/2,\alpha}_{t,x}((0,1)\times B_2)$ for some $0<\alpha\le1$.
		\begin{enumerate}[$(a)$]
			\item If $\alpha+2s$ is not an integer
			then $u$ is in $C^{\alpha/2+s,\alpha+2s}_{t,x}((1/4,3/4)\times B_1)$ with the estimate
			$$\|u\|_{C^{\alpha/2+s,\alpha+2s}_{t,x}((1/4,3/4)\times B_1)}
			\leq C\big(\|f\|_{C^{\alpha/2,\alpha}_{t,x}((0,1)\times B_2)}+\|u\|_{L^\infty((-\infty,1)\times\R^{n})}\big).$$
			\item If $\alpha+2s$ is an integer
			then $u$ is in the parabolic H\"older--Zygmund space 
			$\Lambda^{\alpha/2+s,\alpha+2s}_{t,x}((1/4,3/4)\times B_1)$ with the estimate
			$$\|u\|_{\Lambda^{\alpha/2+s,\alpha+2s}_{t,x}((1/4,3/4)\times B_1)}
			\leq C\big(\|f\|_{C^{\alpha/2,\alpha}_{t,x}((0,1)\times B_2)}+\|u\|_{L^\infty((-\infty,1)\times\R^{n})}\big).$$
		\end{enumerate}
		The constants $C$ in $(a)$ and $(b)$ depend only on dimension, $\alpha$ and $s$.		
	\item Let $f\in L^\infty((0,1)\times B_2)$.
	\begin{enumerate}[$(i)$]
		\item If $0<s<1/2$ then $u\in C^{s,2s}_{t,x}((1/4,3/4)\times B_1)$ with the estimate
		$$\|u\|_{C^{s,2s}_{t,x}((1/4,3/4)\times B_1)}\leq C\big(\|f\|_{L^\infty((0,1)\times B_2)}
		+\|u\|_{L^\infty((-\infty,1)\times\R^{n})}\big).$$
		\item If $s=1/2$ then $u$ is in the parabolic H\"older-Zygmund space
		$\Lambda^{1/2,1}_{t,x}((1/4,3/4)\times B_1)$ with the estimate
		$$\|u\|_{\Lambda^{1/2,1}_{t,x}((1/4,3/4)\times B_1)}\leq C\big(\|f\|_{L^\infty((0,1)\times B_2)}
		+\|u\|_{L^\infty((-\infty,1)\times\R^{n})}\big).$$
		\item If $1/2<s<1$ then $u\in C^{s,1+(2s-1)}_{t,x}((1/4,3/4)\times B_1)$ with the estimate
		$$\|u\|_{C^{s,1+(2s-1)}_{t,x}((1/4,3/4)\times B_1)}\leq C\big(\|f\|_{L^\infty((0,1)\times B_2)}
		+\|u\|_{L^\infty((-\infty,1)\times\R^{n})}\big).$$
	\end{enumerate}
	The constants $C$ in $(i)$, $(ii)$ and $(iii)$ depend only on dimension and $s$.
\end{enumerate}}
\end{thm}

\begin{rem}
As we said before, one of the main purposes of this paper is to establish the regularity
theory for the fractional nonlocal parabolic equation $(\partial_t-\Delta)^su=f$
under the light given by our novel combination of the 
Cauchy Integral Theorem, our original \textit{parabolic language of semigroups} and PDE techniques.
These ideas avoid the heavy use of the special symmetries of the Fourier transform
as in classical approaches. The outstanding advantage here is that
the method we develop has wide applicability. Indeed, we can consider equations like
$(\partial_t-L)^su(t,x)=f(t,x)$,
where $L$ is an elliptic operator in a domain $\Omega$ with boundary conditions,
or the Laplacian in a Riemannian manifold, or the discrete Laplacian on a mesh of
size $h$, just to mention a few.
As we mentioned before, these equations come from applications
in Engineering, Biology, Stochastic Analysis and Numerical Analysis.
It is evident that the results of this paper will be useful to establish
regularity estimates of nonlinear nonlocal space-time equations like the ones in \cite{Caffarelli-Silvestre Master}.
Though applications of these novel ideas will appear elsewhere \cite{BS}, we
stress some further results in Remarks \ref{rem:1}, \ref{rem:2} and \ref{rem:3}.
\end{rem}

The rest of the paper is devoted to the proofs of all the results presented in this Introduction.

\section{The pointwise formula}\label{Section:puntual}

In this section we prove Theorem \ref{thm:puntual} and Corollary \ref{cor:separate}.

 {As the operators $\partial_t$ and $\Delta$ commute, the semigroup $\{e^{-\tau H}\}_{\tau\geq0}$
generated by $H=\partial_t-\Delta$ is given by the composition
$e^{-\tau H}=e^{-\tau \partial_t}\circ e^{\tau\Delta}$.
In particular, for smooth functions $\varphi(t,x)$ with rapid decay at infinity we have
$$e^{-\tau H}\varphi(t,x)=e^{\tau\Delta}\varphi(t-\tau,x)=\int_{\R^n}W(\tau,z)\varphi(t-\tau,x-z)\,dz,$$
where $W(\tau,z)$ denotes the usual Gauss--Weierstrass kernel
\begin{equation}\label{GW}
W(\tau,z):= \frac1{(4\pi\tau)^{n/2}}e^{-|z|^2/(4\tau)}.
\end{equation}
Recall that $\partial_\tau W-\Delta_zW=0$, for $\tau>0$ and $z\in\R^n$.
Notice that for functions $u\in L^2(\R^{n+1})$ we have the Fourier transform identity
$$\widehat{e^{-\tau H}u}(\rho,\xi)= e^{-\tau(i\rho+|\xi|^2)}\widehat{u}(\rho,\xi).$$
In particular,
\begin{equation}\label{eq:decaimiento en L2}
\|e^{-\tau H}u\|_{L^2(\R^{n+1})}\leq\|u\|_{L^2(\R^{n+1})},\quad\hbox{for all}~\tau>0.
\end{equation}}
We shall use the following absolutely convergent integral involving the Gamma function:
$$(i\rho+|\xi|^2)^{s} = \frac1{\Gamma(-s)} \int_0^\infty (e^{-\tau(i\rho+|\xi|^2)}-1)\,\frac{d\tau}{\tau^{1+s}}.$$
The identity above is valid for $\rho\in\R$, $\xi\in\R^n$ and $0<s<1$.
Indeed, it follows from the case $\rho=0$,
$|\xi|=1$, by using the Cauchy Integral Theorem and the analytic continuation of the
function $\tau\longmapsto \tau^s$ from the half line $[0,\infty)$ to the half plane $\Re(z)>0$.
Let us take a function $u=u(t,x)$ such that $(i\rho+|\xi|^2)^{s}\widehat{u}(\rho,\xi)$ is in $L^2(\R^{n+1})$.
Then
$$(i\rho+|\xi|^2)^s\widehat{u}(\rho,\xi)=\frac{1}{\Gamma(-s)}
\int_0^\infty\big(e^{-\tau(i\rho+|\xi|^2)}\widehat{u}(\rho,\xi)-\widehat{u}(\rho,\xi)\big)\,\frac{d\tau}{\tau^{1+s}}.$$
Hence, by Plancherel Theorem,
\begin{equation}\label{eq:semigroup formula}
H^su(t,x)=\frac1{\Gamma(-s)}\int_0^\infty\big(e^{-\tau H}u(t,x)-u(t,x)\big) \, \frac{d\tau}{\tau^{1+s}},
\end{equation}
in $L^2(\R^{n+1})$.
 {Using the Fourier transform we observe that} $v(\tau,t,x)=e^{-\tau H}u(t,x)$ solves
\begin{equation}\label{eq:heatheat}
\begin{cases}
\partial_\tau v=-Hv,&\hbox{for}~\tau>0,~(t,x)\in\R^{n+1},\\
v(0,t,x)=u(t,x),&\hbox{for}~(t,x)\in\R^{n+1}.
\end{cases}
\end{equation}
 {in the $L^2$ sense}. Moreover, as noted above,
\begin{equation}\label{eq:el calor}
v(\tau,t,x)=e^{-\tau H}u(t,x)=\frac{1}{(4\pi\tau)^{n/2}}\int_{\R^n}e^{-|x-z|^2/(4\tau)}u(t-\tau,z)\,dz,
\end{equation}
for $\tau>0$, $(t,x)\in\R^{n+1}$.

\begin{defn}[Parabolic H\"older spaces, see also \cite{Krylov}]
Let $u=u(t,x)$ be a bounded continuous function in $\R^{n+1}$. 
\begin{enumerate}
\item Suppose that $0<\alpha\leq1$.
Then $u$ is in the parabolic H\"older space $C^{\alpha/2,\alpha}_{t,x}(\mathbb{R}^{n+1})$
if there exists a positive constant $C$ such that
$$|u(t-\tau,x-z)-u(t,x)|\leq C(|\tau|^{1/2}+|z|)^{\alpha},$$
for every $t,\tau\in\R$, $x,z\in\R^n$. The least constant $C$ for which the inequality above is true
is denoted by $[u]_{C^{\alpha/2,\alpha}_{t,x}}$.

\item For $1<\alpha\leq2$ we say that
$u$ is in the space $C^{\alpha/2,\alpha}_{t,x}(\R^{n+1})=C^{\alpha/2,1+(\alpha-1)}_{t,x}(\R^{n+1})$
if $u$ is $\alpha/2$-H\"older continuous in $t$ uniformly in $x$
and its gradient $\nabla_xu$ is $(\alpha-1)$-H\"older continuous in $x$ uniformly in $t$.

\item In general, for any $\alpha>0$, we say that $u$ is in the space
$$C^{\alpha/2,\alpha}_{t,x}(\R^{n+1})=C^{[\alpha/2]+(\alpha/2-[\alpha/2]),[\alpha]+(\alpha-[\alpha])}_{t,x}(\R^{n+1}),$$
if $u$ has $[\alpha/2]$ time derivatives $(\alpha/2-[\alpha/2])$-H\"older continuous uniformly in $x$
and $[\alpha]$ space derivatives $(\alpha-[\alpha])$-H\"older continuous uniformly in $t$. Here $[\alpha]$ denotes
the integer part of $\alpha$.
\end{enumerate}
The norms in $C^{\alpha/2,\alpha}_{t,x}(\R^{n+1})$ are given in the usual way.
We can also define the spaces $C^{\alpha/2,\alpha}_{t,x}(D)$, where $D\subset\R^{n+1}$
is an open set, by restricting all the definitions above to $(t,x),(\tau,z)\in D$.
\end{defn}

\begin{proof}[Proof of Theorem \ref{thm:puntual}]
Suppose  that $u\in C^{s+\varepsilon,2s+\varepsilon}_{t,x}(\mathbb{R}^{n+1})$,
for some $\varepsilon>0$.
As $e^{-\tau H}1\equiv1$, from \eqref{eq:semigroup formula} and \eqref{eq:el calor} we can write
$$H^su(t,x)=\frac1{\Gamma(-s)}\int_0^\infty\int_{\mathbb{R}^n} 
\frac{e^{-|z|^2/(4\tau)}}{(4 \pi \tau)^{n/2}} \big(u(t-\tau,x-z)-u(t,x)\big)\,dz\, \frac{d\tau}{\tau^{1+s}}.$$
It is clear that for $\tau>0$ and $z\in\R^n$,
$$\frac{e^{-|z|^2/(4\tau)}} {\tau^{n/2+1+s}}  \le\frac{\Lambda}{|z|^{n+2+2s}+\tau^{(n+2+2s)/2}},$$
for some positive constant $\Lambda$,
so estimate $(2)$ in the statement follows. Also, $(1)$ is clearly true.
By using the regularity of $u$, the integral above is absolutely convergent near the singularity.
Moreover, the integral converges at infinity because $u$ is bounded.
\end{proof}

\begin{proof}[Proof of Corollary \ref{cor:separate}]
Suppose that $u$ does not depend on $t$. Then, by using the change of variables $r=|z|^2/(4\tau)$
(see the details in \cite{Stinga Thesis, Stinga-Torrea CPDE}),
\begin{align*}
H^su(t,x) &= \frac{1}{|\Gamma(-s)|}\PV\int_{\R^n}(u(x)-u(x-z))\int_0^\infty
\frac{e^{-|z|^2/(4\tau)}}{(4\pi\tau)^{n/2}}\,\frac{d\tau}{\tau^{1+s}}\,dz \\
&= \frac{4^s\Gamma(n/2+s)}{\pi^{n/2}|\Gamma(-s)|}\PV\int_{\R^n}\frac{u(x)-u(x-z)}{|z|^{n+2s}}\,dz=(-\Delta)^su(x),
\end{align*}
where $\PV$ means that the integral in $dz$ is taken in the principal value sense.
On the other hand, if $u$ does not depend on $x$ then
\begin{align*}
H^su(t,x) &= \frac{1}{|\Gamma(-s)|}\int_0^\infty(u(t)-u(t-\tau))\int_{\R^n}
\frac{e^{-|z|^2/(4\tau)}}{(4\pi\tau)^{n/2}}\,dz\,\frac{d\tau}{\tau^{1+s}} \\
&= \frac{1}{|\Gamma(-s)|}\int_0^\infty\frac{u(t)-u(t-\tau)}{\tau^{1+s}}\,d\tau=(\partial_t)^su(t),
\end{align*}
as claimed.
\end{proof}

\begin{rem}[Pointwise formula in the general case]\label{rem:general pointwise}
 {Throughout the paper we assume, for the sake of simplicity, that $u\in L^\infty(\R^{n+1})$.
As usual in nonlocal equations, this assumption can be relaxed. Indeed,}
let $u=u(t,x)$ be a function such that
$$\int_0^{\infty}\int_{\R^n}|u(t-\tau,z)|\frac{e^{-|z|^2/(4\tau)}}{(1+\tau^{n/2+1+s})}\,dz\,d\tau<\infty,$$
for every $t\in\R$.
If $u$ satisfies the parabolic H\"older condition $C^{s+\varepsilon,2s+\varepsilon}_{t,x}$,
for some $\varepsilon>0$ in some open set $D\subset\R^{n+1}$,
then it can be seen that the pointwise formula for $H^su(t,x)$ in
\eqref{eq:pointwise formula} is valid for every $(t,x)\in D$. This follows by a standard approximation
procedure, see for instance
\cite{Silvestre Thesis, Stinga Thesis, Stinga-Torrea JFA}. 
 {In particular, $H^su(t,x)$ can be written with the pointwise formula
for $(t,x)\in D$ if we assume the above-mentioned parabolic H\"older regularity in of $u$ in $D$
and we require $u$ to be, say, bounded in $(-\infty,T)\times\R^n$, where $T=\sup\{t:(t,x)\in D\}$.}
Details are left to the interested reader.
\end{rem}

\begin{rem}\label{rem:1}
Let $Lu=\dive(a(x)\nabla u)$ be an elliptic operator in divergence form on
a domain $\Omega\subseteq\R^n$, subject to appropriate boundary conditions.
Using the ideas in this section we see that if $e^{tL}1(x)\equiv 1$ then
the following identity holds in the weak sense:
$$(\partial_t-L)^su(t,x)=\frac{1}{\Gamma(-s)}\int_0^\infty\int_\Omega k_\tau(x,z)(u(t-\tau,z)-u(t,x))\,dz\,\frac{d\tau}{\tau^{1+s}},$$
where $k_\tau(x,z)$ is the heat kernel of $L$.
If $e^{tL}1(x)$ is not identically $1$, then it can be seen that a multiplicative term appears.
When $Lu=\tr(a(x)D^2u)$ is in non divergence form,
parallel identities for $(\partial_t-L)^su(t,x)$ are true in the suitable sense (pointwise, strong, viscosity).
In a similar way we can write the powers
$(\partial_t-\Delta_M)^s$, where $\Delta_M$ is the Laplace--Beltrami operator on a Riemannian manifold $M$.
If $L=\Delta_h$, the discrete Laplacian on $\Z_h=h\Z$, $h>0$, we get
the fractional powers of the semidiscrete heat operator acting on a function
$u_j(t):\Z_h\times\R\to\R$ as
$$(\partial_t-\Delta_h)^su_j(t)=\frac{1}{\Gamma(-s)}\int_0^\infty\sum_{m\in\Z}
G(m,\tfrac{\tau}{h^2})(u_{j-m}(t-\tau)-u_j(t))\,\frac{d\tau}{\tau^{1+s}},$$
where $G(m,t)$ is the semidiscrete heat kernel. When $u$ does 
not depend on $t$ then we recover the fractional power of the discrete Laplacian $(-\Delta_h)^su_j$, see \cite{los5}.
Further development and applications of these ideas will appear elsewhere.
\end{rem}

\section{Maximum and comparison principles -- Uniqueness}

In this section we prove Theorem \ref{thm:maximum} and Corollary \ref{cor:comparison}.

\begin{proof}[Proof of Theorem \ref{thm:maximum}]
If $u$ satisfies $(1)$ and $(2)$ then it follows from
the pointwise formula \eqref{eq:pointwise formula} and the non negativity of the kernel $K_s$
that $H^su(t_0,x_0)\leq 0$. Moreover, if $H^su(t_0,x_0)=0$, then
$u(t,x)$ has to be zero for all $t\leq t_0$ and $x\in\R^n$.
\end{proof}

\begin{lem}\label{lem:premax}
Let $\Omega$ be an bounded domain of $\R^n$ and let $T_0<T_1$ be two real numbers.
 {Let $u\in C^{s+\varepsilon,2s+\varepsilon}_{t,x}(\R^{n+1})$, for some $\varepsilon>0$}.
\begin{enumerate}[$(1)$]
\item Suppose that $u$ solves
$$\begin{cases}
H^su=f,&\hbox{in}~(T_0,T_1]\times\Omega,\\
u=0,&\hbox{in}~D:=\big[(-\infty,T_0]\times\R^n\big]\cup\big[(T_0,T_1)\times(\R^n\setminus\Omega)\big].
\end{cases}$$
If $f\geq0$ in $(T_0,T_1]\times\Omega$ then $u\geq0$ everywhere in $(-\infty,T_1]\times\R^n$.
\item If $H^su\leq0$ in $(T_0,T_1]\times\Omega$ and $u\leq0$ in 
$D=\big[(-\infty,T_0]\times\R^n\big]\cup\big[(T_0,T_1)\times(\R^n\setminus\Omega)\big]$ then
$$\sup_{t<T_1,x\in\mathbb{R}^n}u(t,x)=\sup_{D}u(t,x).$$
\item If $H^su\geq0$ in $(T_0,T_1]\times\Omega$ and $u\geq0$ in
$D=\big[(-\infty,T_0]\times\R^n\big]\cup\big[(T_0,T_1)\times(\R^n\setminus\Omega)\big]$ then
$$\inf_{t<T_1,x \in\R^n}u(t,x)=\inf_{D}u(t,x).$$
\end{enumerate}
\end{lem}

\begin{proof}
For $(1)$, by contradiction, suppose that there
is a point $(t_0,x_0)\in(T_0,T_1]\times\Omega$ where $u$ attains a negative minimum
over $(-\infty,T_1]\times\R^n$.
Then, by Theorem \ref{thm:maximum},
$H^su(t_0, x_0)\leq0$. Now we have two cases.
If $H^su(t_0,x_0)=0$
then necessarily $u(t,x)=u(t_0,x_0)<0$ for all $t\le t_0$, $x\in\R^n$, contradicting the hypothesis that
$u(t,x)=0$ in $D$. The case $H^su(t_0,x_0)<0$
also leads to a contradiction because, by assumption, $f(t_0,x_0)\geq0$.
Therefore $u\geq0$ in $(-\infty,T_1]\times\R^n$.

Notice that (3) follows from (2) by considering $-u$.
For (2), again by contradiction, suppose that $\sup_{t<T_1,x\in\R^n}u(t,x)$
is not attained in $D$.
Then there exists $(t_0,x_0)\in(T_0,T_1]\times\Omega$ such that $u(t_0,x_0)$
is a maximum of $u$ over $(-\infty,T_1]\times\R^n$. Hence, by
Theorem \ref{thm:maximum}, $H^su(t_0,x_0)\geq0$.
The case $H^su(t_0,x_0)>0$ contradicts the fact that $H^su\leq0$
in $(T_0,T_1]\times\Omega$. Now if $H^su(t_0,x_0)=0$ then,
by Theorem \ref{thm:maximum}, $u(t,x)=u(t_0,x_0)$ for all $t\le t_0$, $x\in\mathbb{R}^n$.
This implies that $\sup_{t<T_1,x\in\R^n}u(t,x)$ is attained in $D$,
again a contradiction with our
initial assumption.
\end{proof}

\begin{proof}[Proof of Corollary \ref{cor:comparison}]
The statement follows by applying Lemma \ref{lem:premax}$(3)$ to $u-w$. The uniqueness is immediate.
\end{proof}

\section{The parabolic extension problem}

In this section we prove Theorems \ref{thm:extension} and \ref{thm:fundamental}.

\begin{proof}[Proof of Theorem \ref{thm:extension}]
Since $U$ is given explicitly, we just have to make sure that it is well defined, satisfies the equation,
attains the boundary data and verifies the Neumann boundary condition.

It is easy to check that, for every $y>0$,
\begin{equation}\label{integral}
\frac{y^{2s}}{4^s\Gamma(s)}\int_0^\infty e^{-y^2/(4\tau)}\,\frac{d\tau}{\tau^{1+s}}=1.
\end{equation}
 {By using \eqref{eq:decaimiento en L2} and \eqref{integral}, for any $y>0$,
$$\|U(\cdot,\cdot,y)\|_{L^2(\R^{n+1})}
\leq \frac{y^{2s}}{4^s\Gamma(s)}\int_0^\infty
e^{-y^2/(4\tau)}\|e^{-\tau H}u\|_{L^2(\R^{n+1})}\,\frac{d\tau}{\tau^{1+s}}
\leq \|u\|_{L^2(\R^{n+1})},$$
and thus
\begin{equation}\label{eq:UFourier}
\widehat{U}(\rho,\xi,y)= \frac{y^{2s}}{4^s\Gamma(s)}\int_0^\infty e^{-y^2/(4\tau)}e^{-\tau(i\rho+|\xi|^2)}
\widehat{u}(\rho,\xi)\,\frac{d\tau}{\tau^{1+s}}.
\end{equation}
It is easily checked that $U$ can be differentiated
in $y$ by taking the derivative inside the integral sign, so
$U(\cdot,\cdot,y)\in C^\infty((0,\infty);L^2(\R^{n+1}))$. Additionally,
$e^{-\tau H}u\in\Dom(H)$ and henceforth $U(\cdot,\cdot,y)\in C^\infty((0,\infty);\Dom(H))$.}

We next show that all the expressions for $U$ given in the statement are equivalent.
 The second identity in \eqref{eq:U} follows from the first one by the change of variables
$r=y^2/(4\tau)$, while the third one follows again from the first one but now
via the change of variables
 {$r= y^2/(4\tau(i\rho+|\xi|^2))$ in the Fourier transform side \eqref{eq:UFourier}.}

Now let us prove that $U$ satisfies all the properties of the statement in the $L^2$ sense.
Indeed, from the first identity for $U$ in \eqref{eq:U}  and using the fact that
$e^{-\tau H}u$ solves \eqref{eq:heatheat}, by integrating by parts we get 
\begin{align*}
(\partial_t-\Delta)U &= \frac{y^{2s}}{4^s\Gamma(s)}\int_0^\infty e^{-y^2/(4\tau)}H\big(e^{-\tau H}u(t,x)\big)\,\frac{d\tau}{\tau^{1+s}} \\
&= -\frac{y^{2s}}{4^s\Gamma(s)}\int_0^\infty e^{-y^2/(4\tau)}\partial_\tau\big(e^{-\tau H}u(t,x)\big)\,\frac{d\tau}{\tau^{1+s}} \\
&= \frac{y^{2s}}{4^s\Gamma(s)}\int_0^\infty \partial_\tau\bigg[\frac{e^{-y^2/(4\tau)}}{\tau^{1+s}}\bigg]
e^{-\tau H}u(t,x)\,d\tau \\
&= \frac{1}{4^s\Gamma(s)}\int_0^\infty\Big(\tfrac{1-2s}{y}\partial_y+\partial_{yy}\Big)
\bigg[\frac{y^{2s}e^{-y^2/(4\tau)}}{\tau^{1+s}}\bigg]e^{-\tau H}u(t,x)\,d\tau \\
&= \Big(\tfrac{1-2s}{y}\partial_y+\partial_{yy}\Big)U.
\end{align*}
 {All the identities above are understood in $L^2(\mathbb{R}^{n+1})$, that is,
through the Fourier transform identity \eqref{eq:UFourier}.}
From the second identity for $U$ in \eqref{eq:U} and by using that $e^{-\tau H}u|_{\tau=0}=u$
and the definition of the Gamma function, it is obvious
that $U$ satisfies the boundary condition  {$\lim_{y\to0^+}U(t,x,y)=u(t,x)$ in $L^2(\R^{n+1})$}.

To prove the first Neumann
boundary condition in \eqref{eq:Neumann}, consider the first identity for $U$ in \eqref{eq:U}.
By \eqref{integral} we can write
$$\frac{U(t,x,y)-U(t,x,0)}{y^{2s}}= \frac{1}{4^s\Gamma(s)}\int_0^\infty e^{-y^2/(4\tau)}
\big(e^{-\tau H}u(t,x)-u(t,x)\big)\,\frac{d\tau}{\tau^{1+s}},$$
which, as $y\to0^+$, clearly converges to (see \eqref{eq:semigroup formula})
$$\frac{1}{4^s\Gamma(s)}\int_0^\infty\big(e^{-\tau H}u(t,x)-u(t,x)\big)\,\frac{d\tau}{\tau^{1+s}}
=\frac{\Gamma(-s)}{4^s\Gamma(s)}H^su(t,x),$$
in $L^2(\R^{n+1})$.
In a similar way, we use the derivative
with respect to $y$ of \eqref{integral} to write
$$y^{1-2s}\partial_yU(t,x,y)=\frac{1}{4^s\Gamma(s)}\int_0^\infty \bigg(2s-\frac{y^2}{2\tau}\bigg)e^{-y^2/(4\tau)}
\big(e^{-\tau H}u(t,x)-u(t,x)\big)\,\frac{d\tau}{\tau^{1+s}}.$$
Then the second Neumann condition in \eqref{eq:Neumann} follows.

Finally,
the Poisson formula \eqref{eq:Poisson formula} is a consequence of the first identity for $U$ in \eqref{eq:U}
by writing $e^{-\tau H}u(t,x)$ with a kernel as in \eqref{eq:el calor}.
\end{proof}

\begin{proof}[Proof of Theorem \ref{thm:fundamental}]
That $\mathcal{U}(t,x,y)$ solves the equation in $L^2(\mathbb{R}^{n+1})$
can be seen just by a direct computation via Fourier transform.
If $f$ is good enough the performed computation is classical.
\end{proof}

\begin{rem}\label{rem:2}
As in Remark \ref{rem:1},
we stress that our ideas can be used in other important contexts. Let $Lu=\dive(a(x)\nabla u)$ be an elliptic operator 
in divergence form on a domain $\Omega\subseteq\R^n$
subject to suitable boundary conditions. If $U(t,x,y)$ is a weak solution to the
degenerate parabolic extension problem
$$\begin{cases}
\displaystyle\partial_tU=y^{-(1-2s)}\dive_{x,y}(y^{1-2s}B(x)\nabla_{x,y}U),&\hbox{for}~t\in\R,~x\in\Omega,~y>0,\\
\displaystyle U(t,x,0)=u(t,x),&\hbox{for}~t\in\R,~x\in\Omega,
\end{cases}$$
subject to appropriate boundary conditions, where
$$B(x):=
\begin{pmatrix}
a(x) & 0 \\
0 & 1
\end{pmatrix}
\in  {\R^{(n+1)\times(n+1)}},$$
then
$$-y^{1-2s}\partial_yU(t,x,y)\big|_{y=0}=c_s(\partial_t-L)^su(t,x),$$
in the weak sense. Similarly we can establish extension results for the non divergence form case $Lu=\tr(a(x)D^2u)$,
the Riemannian manifold case $Lu=\Delta_Mu$, or the discrete case $L=\Delta_h$. Further 
developments of these ideas and applications will appear elsewhere \cite{BS}.
\end{rem}

\section{{Harnack inequalities and interior estimates}}\label{Section:Harnack}

In this section we prove Theorem \ref{thm:Harnack}, Corollary \ref{cor:interior}
and Theorem \ref{thm:boundary Harnack}.

We need the following preliminary result.

\begin{lem}[Reflection extension]\label{lem:reflection}
Let $\Omega$ be an open subset of $\R^{n}$, $(T_0,T_1)$ an open interval in $\R$ and $Y_0>0$. 
Suppose that a function $U=U(t,x,y):(T_0,T_1)\times\Omega\times(0,Y_0)\to\R$ satisfies
$$\partial_tU=y^{-(1-2s)}\dive_{x,y}(y^{1-2s}\nabla_{x,y}U),$$
for $(t,x)\in(T_0,T_1)\times\Omega,\,0<y<Y_0$, with 
$$\lim_{y\to0^+}y^{1-2s}U_y(t,x,y)=0,$$
for $(t,x)\in(T_0,T_1)\times\Omega$, in the weak sense.  {In other words,
for every smooth test function $\varphi=\varphi(t,x,y)$ with compact support in $(T_0,T_1)\times\Omega\times[0,Y_0)$,
$$\int_\Omega\int_0^{Y_0}\int_{T_0}^{T_1}U(\partial_t\varphi)y^{1-2s}\,dt\,dy\,dx
=\int_\Omega\int_0^{Y_0}\int_{T_0}^{T_1}(\nabla_{x,y}U)(\nabla_{x,y}\varphi)y^{1-2s}\,dt\,dy\,dx,$$
and
\begin{equation}\label{eq:a cero}
\lim_{y\to0^+}\int_\Omega\int_{T_0}^{T_1}y^{1-2s}U_y\varphi\,dt\,dx=0.
\end{equation}}
Then the even extension of $U$ in $y$ defined by
\begin{equation}\label{eq:reflection}
\widetilde{U}(t,x,y):=U(t,x,|y|),\quad\hbox{for}~y\in(-Y_0,Y_0),
\end{equation}
verifies the degenerate parabolic equation
\begin{equation}\label{eq:weak across}
\partial_t\widetilde{U}=|y|^{-(1-2s)}\dive_{x,y}(|y|^{1-2s}\nabla_{x,y}\widetilde{U}),
\end{equation}
in the weak sense in $(T_0,T_1)\times\Omega\times(-Y_0,Y_0)$.  {In other words, for every smooth
test function $\phi=\phi(t,x,y)$ with compact support in $(T_0,T_1)\times\Omega\times(-Y_0,Y_0)$,
$$\int_\Omega\int_{-Y_0}^{Y_0}\int_{T_0}^{T_1}\widetilde{U}(\partial_t\phi)|y|^{1-2s}\,dt\,dy\,dx
=\int_\Omega\int_{-Y_0}^{Y_0}\int_{T_0}^{T_1}(\nabla_{x,y}\widetilde{U})(\nabla_{x,y}\phi)|y|^{1-2s}\,dt\,dy\,dx.$$}
\end{lem}

\begin{proof}
Let $\phi$ be a smooth test function with compact support in
$(T_0,T_1)\times\Omega\times(-Y_0,Y_0)$. For any $\delta>0$,
\begin{align*}
\int_\Omega&\int_{T_0}^{T_1}\int_{-Y_0}^{Y_0}(\nabla_{x,y}\widetilde{U})(\nabla_{x,y}\phi)|y|^{1-2s}\,dy\,dt\,dx= \\
&= \int_\Omega\int_{T_0}^{T_1}\int_{|y|>\delta}\dive_{x,y}(|y|^{1-2s}\phi\nabla_{x,y}\widetilde{U})\,dy\,dt\,dx 
+\int_\Omega\int_{T_0}^{T_1}\int_{|y|>\delta}\widetilde{U}(\partial_t\phi)|y|^{1-2s}\,dy\,dt\,dx \\
&\quad+\int_\Omega\int_{T_0}^{T_1}\int_{|y|<\delta}(\nabla_{x,y}\widetilde{U})(\nabla_{x,y}\phi)|y|^{1-2s}\,dy\,dt\,dx \\
&= \int_\Omega\int_{T_0}^{T_1}(|y|^{1-2s}\phi\widetilde{U}_y)\big|_{y=\delta}\,dt\,dx 
+\int_\Omega\int_{T_0}^{T_1}\int_{|y|>\delta}\widetilde{U}(\partial_t\phi)|y|^{1-2s}\,dy\,dt\,dx \\
&\quad+\int_\Omega\int_{T_0}^{T_1}\int_{|y|<\delta}(\nabla_{x,y}\widetilde{U})(\nabla_{x,y}\phi)|y|^{1-2s}\,dy\,dt\,dx.
\end{align*}
Taking $\delta\to0$, we see that the first and third integrals above
tend to zero because of \eqref{eq:a cero} and the local integrability of the integrand, respectively.
Thus, $\widetilde{U}$ is a weak solution to \eqref{eq:weak across}.
\end{proof}

\begin{proof}[Proof of Theorem \ref{thm:Harnack}]
Let $u$ be as in the statement and consider its extension $U$ for $y>0$ as in Theorem \ref{thm:extension}.
Notice that, since $u(t,x)\geq0$ for $t<1$ and $x\in\R^n$, by \eqref{eq:Poisson formula}--\eqref{eq:PoissonKernel}
we have $U=U(t,x,y)\geq0$ for all $0<t<1$, $x\in B_2$ and $y>0$. Let $\widetilde{U}$ be the reflection of $U$ as
given by \eqref{eq:reflection}. Then by Lemma \ref{lem:reflection}, $\widetilde{U}$ is a nonnegative weak solution
to the degenerate parabolic equation \eqref{eq:weak across}
in $(t,x,y)\in(0,1)\times B_2\times(-2,2)$.For this equation the Harnack inequality of
Chiarenza--Serapioni \cite{Chiarenza-Serapioni} (see also Guti\'errez--Wheeden \cite{Gutierrez-Wheeden},
Ishige \cite{Ishige}) applies. Therefore
\begin{align*}
\sup_{(t,x)\in R^-}u(t,x) &= \sup_{(t,x)\in R^-}U(t,x,0)\leq\sup_{(t,x,y)\in R^-\times(-1,1)}U(t,x,y) \\
&\leq c\inf_{(t,x,y)\in R^+\times(-1,1)}U(t,x,y)\leq c\inf_{(t,x)\in R^+}U(t,x,0)=c\inf_{(t,x)\in R^+}u(t,x),
\end{align*}
where $c>0$ is the constant for the Harnack inequality of Chiarenza--Serapioni \cite{Chiarenza-Serapioni},
so it depends only on $n$ and $s$. 
\end{proof}

\begin{proof}[Proof of Corollary \ref{cor:interior}]
 {Let us first prove $(1)$. Let $\widetilde{U}$ be the reflection in $y$ of the extension $U$ of $u$
as in the proof of Theorem \ref{thm:Harnack} above.
Since $\widetilde{U}$ satisfies \eqref{eq:weak across} in the weak sense in $(t,x,y)\in(0,1)\times B_2\times(-2,2)$, the 
Harnack inequality of Chiarenza--Serapioni implies that $\widetilde{U}$ is parabolically H\"older continuous
of order $0<\alpha=\alpha(n,s)<1$ in $(1/2,1)\times B_1\times(-1,1)$ (this argument is standard, see Moser \cite{Moser}).
Therefore, $u(t,x)=\widetilde{U}(t,x,0)$ is parabolically H\"older continuous in $(1/2,1)\times B_1$.
The estimate in the statement follows by observing that (see also \eqref{eq:Poisson formula}--\eqref{eq:PoissonKernel})
\begin{align*}
\sup_{y\in(-2,2)}\|\widetilde{U}\|_{L^\infty((0,1)\times B_2)}
&\leq \sup_{y\in(0,2)}\frac{2y^{2s}}{4^s\Gamma(s)}\int_0^\infty e^{-y^2/(4\tau)}
\sup_{\tau>0}\|e^{-\tau H}u\|_{L^\infty((0,1)\times B_2)}\,\frac{d\tau}{\tau^{1+s}}\\
&\leq\|u\|_{L^\infty((-\infty,1)\times\R^n)}\sup_{y\in(0,2)}\frac{2y^{2s}}{4^s\Gamma(s)}\int_0^\infty e^{-y^2/(4\tau)}\,\frac{d\tau}{\tau^{1+s}} \\
&=2\|u\|_{L^\infty((-\infty,1)\times\R^n)},
\end{align*}
where in the last identity we used \eqref{integral}.}

 {For part $(2)$ we use an incremental quotient argument devised by Caffarelli for fully
nonlinear elliptic equations (see the proof of \cite[Corollary~5.7]{Caffarelli-Cabre}).
Recall that the reflection extension $\widetilde{U}$ of $u$ is a weak solution to
\begin{equation}\label{eq:misma}
\partial_t\widetilde{U}=|y|^{-(1-2s)}\dive_{x,y}(|y|^{1-2s}\nabla_{x,y}\widetilde{U}),
\end{equation}
in $(0,1)\times B_2\times(-2,2)$. It is enough to prove that $\widetilde{U}$
is differentiable with respect to $t$ and $x$, with
$$\sup_{(1/4,3/4)\times B_{1}\times(-1,1)}|\partial^k_tD^\beta_x\widetilde{U}(t,x,y)|
\leq C\|\widetilde{U}\|_{L^\infty((0,1)\times B_2\times(-2,2))},$$
and then restrict back to $y=0$. We know from Harnack inequality that there exists $\alpha=\alpha(n,s)>0$ such that
$$\|\widetilde{U}\|_{C^{\alpha/2,\alpha}_{t,(x,y)}((1/8,7/8)\times B_{3/2}\times(-3/2,3/2))}
\leq C\|\widetilde{U}\|_{L^\infty((0,1)\times B_2\times(-2,2))}.$$
Let $e$ be a unit vector in $\R^n$. Then, by the estimate above, the incremental quotient of order $\alpha$
$$V_h(t,x,y)=\frac{\widetilde{U}(t,x+he,y)-\widetilde{U}(t,x,y)}{|h|^\alpha}$$
is bounded in $(1/8,7/8)\times B_{3/2}\times(-3/2,3/2)$ independently of $h$. In addition, since the extension equation
is linear and translation invariant with respect to $x$, we see that $V_h$ is also a solution to 
\eqref{eq:misma}. Therefore, $V_h$ is parabolically $C^\alpha$ and, in particular, is $C^\alpha$ in $x$.
But then, by \cite[Lemma~5.6]{Caffarelli-Cabre}, $\widetilde{U}$ is $C^{2\alpha}$ in $x$.
By iteration of this argument a finite number of times (that is, by considering incremental
quotients of order $2\alpha$, $3\alpha$ and so on)
we obtain that $\widetilde{U}$ is Lipschitz and, finally, $C^{1,\alpha}$
in $x$ (see \cite[Lemma~5.6]{Caffarelli-Cabre}).
At this point we can differentiate the equation with respect to $x$ and bootstrap the above reasoning
to obtain that $\widetilde{U}$ is differentiable with respect to $x$. Exactly the same arguments can be used
starting from the incremental quotient of order $\alpha/2$ in $t$ to show that 
$\widetilde{U}$ is smooth in $t$. The proof is completed}.
\end{proof}

\begin{proof}[Proof of Theorem \ref{thm:boundary Harnack}]
Let $u$ be as in the statement and consider its extension $U$ as in Theorem \ref{thm:extension}.
As in the previous proofs,
$U=U(t,x,y)\geq0$ for $-2<t<2$, $x\in B_2$ and $y>0$, and its 
reflection $\widetilde{U}$ is a nonnegative weak solution
to \eqref{eq:weak across} in $(-2,2)\times\Omega\times\R$.
We also have that $\widetilde{U}(t,x,y)$ vanishes continuously in
$(-2,2)\times\big((\R^n\setminus\Omega)\cap B_2\big)\times\{0\}$, because
in that set it coincides with $u$.

The first step is to flatten the boundary of $\Omega$ inside $B_2$ by using a bilipschitz transformation
that can be extended as constant in $t$ and $y$. We can assume without loss of generality
that the flat part of the upper half ball $\R^n_+\cap B_{2}$ is the flat part of the new domain $\Omega^{(1)}$.
Then the transformed solution $\widetilde{U}^{(1)}$
satisfies now the same type of degenerate parabolic equation with bounded measurable coefficients
in $(-2,2)\times(\R^n_+\cap B_{2})\times\R$. Moreover, $\widetilde{U}^{(1)}$
vanishes continuously on $(-2,2)\times\big(\R^n_{-}\cap B_{2}\big)\times\{0\}$.

Next we take a transformation
that maps $(x',x_n,y)\in\R^{n+1}\setminus\{x_n\leq0,y=0\}$ into the half space $\R^{n+1}\cap\{x_n>0\}$
and is extended to be constant in $t$.
The construction of this transformation is by now standard, for details see
for example \cite{Caffarelli-Silvestre CPDE} or \cite{Roncal-Stinga}.
After this transformation is performed, we obtain a function $\widetilde{U}^{(2)}$
that solves again a degenerate parabolic equation with bounded measurable coefficients
and $A_2$-weight for $(t,x,y)\in(-2,2)\times\big(\R^n_+\cap B_{2}\big)\times\R$
and vanishes continuously for $(t,x,y)\in(-2,2)\times(\partial\R^n_+\cap B_{2})\times(-2,2)$.

We can apply the boundary Harnack inequality of Ishige \cite{Ishige}
(or Salsa \cite{Salsa} in the case $s=1/2$) to $\widetilde{U}^{(2)}$
and the conclusion follows by transforming back to $u$. Indeed,
$$\sup_{(-1,1)\times(\Omega\cap B_1)}u(t,x)= \sup_{(-1,1)\times(\R^n_+\cap B_1)}\widetilde{U}^{(2)}(t,x,0)
\leq C\widetilde{U}^{(2)}(t_0,\tilde{x}_0,0)= Cu(t_0,x_0),$$
where $\tilde{x}_0$ is the point obtained from $x_0$ after the two transformations above are applied.
\end{proof}

\begin{rem}\label{rem:3}
Recall Remark \ref{rem:2}.
Our combination of ideas and techniques show that Theorems \ref{thm:Harnack} and  \ref{thm:boundary Harnack} are true
when the Laplacian $\Delta$ is replaced by a symmetric uniformly elliptic operator $Lu=\dive(a(x)\nabla u)$
in divergence form with bounded measurable coefficients
in a domain $\Omega\subseteq\R^n$, subject to appropriate boundary conditions.
\end{rem}

\section{Almgren frequency formula}\label{Section:Almgren}

In this section we prove Theorem \ref{thm:Almgren}.

Recall the notation established right before the statement of Theorem \ref{thm:Almgren}.

\begin{proof}[Proof of Theorem \ref{thm:Almgren}]
 {We will repeatedly use the following simple facts:
$$\nabla\G_s=\frac{X}{2t}\G_s,\qquad y^a\partial_t\G_s=-\dive(y^a\nabla\G_s).$$ 
We can write, by using integration by parts,
\begin{align*}
I(R) &= R^2\int_{\R^{n+1}_+}y^a\nabla U\nabla U\G_s\,dX \\
&= -R^2\int_{\R^{n+1}_+}U\dive(y^a\nabla U\G_s)\,dX
+\lim_{\varepsilon\to0^+}\int_{\R^n}y^aUU_y\G_s\big|_{y=\varepsilon}\,dx \\
&= -R^2\int_{\R^{n+1}_+}y^aU\Big(\partial_tU+\nabla U\frac{X}{2t}\Big)\G_s\big|_{t=-R^2}\,dX.
\end{align*}
Therefore, integrating by parts,
\begin{align*}
H'(R) &= -2R\int_{\R^{n+1}_+}\big(y^a2U\partial_tU\G_s+U^2y^a\partial_t\G_s\big)\,dX \\
&= -2R\int_{\R^{n+1}_+}\big(y^a2U\partial_tU\G_s-U^2\dive(y^a\nabla\G_s)\big)\,dX \\
&=-4R\int_{\R^{n+1}_+}y^aU\big(\partial_tU\G_s+\nabla U\nabla\G_s\big)\,dX \\
&=-4R\int_{\R^{n+1}_+}y^aU\Big(\partial_tU+\nabla U\frac{X}{2t}\Big)\G_s\big|_{t=-R^2}\,dX=\frac{4}{R}I(R).
\end{align*}
Next we compute
\begin{align*}
I'(R) &= 2R\int_{\R^{n+1}_+}y^a|\nabla U|^2\G_s\,dX-2R^3\int_{\R^{n+1}_+}
\big(y^a2\nabla U\nabla U_t\G_s+|\nabla U|^2y^a\partial_t\G_s\big)\,dX \\
&= 2R\int_{\R^{n+1}_+}y^a|\nabla U|^2\G_s\,dX-2R^3\int_{\R^{n+1}_+}
\big(y^a2\nabla U\nabla U_t\G_s-|\nabla U|^2\dive(y^a\nabla\G_s)\big)\,dX \\
&= 2R\int_{\R^{n+1}_+}y^a|\nabla U|^2\G_s\,dX-4R^3\int_{\R^{n+1}_+}
y^a\big(\nabla U\nabla U_t\G_s+\nabla U D^2U\nabla\G_s\big)\,dX \\
&= 2R\int_{\R^{n+1}_+}y^a|\nabla U|^2\G_s\,dX-4R^3\int_{\R^{n+1}_+}
y^a\nabla U\Big(\nabla U_t+D^2U\frac{X}{2t}\Big)\G_s\big|_{t=-R^2}\,dX \\
&=-4R^3\int_{\R^{n+1}_+}y^a\nabla U\nabla\Big(\partial_tU+\nabla U\frac{X}{2t}\Big)\G_s\big|_{t=-R^2}\,dX \\
&=4R^3\int_{\R^{n+1}_+}\big(\dive(y^a\nabla U)\G_s
+y^a\nabla U\nabla\G_s\big)\Big(\partial_tU+\nabla U\frac{X}{2t}\Big)\big|_{t=-R^2}\,dX \\
&=4R^3\int_{\R^{n+1}_+}y^a\Big(\partial_tU+\nabla U\frac{X}{2t}\Big)^2\G_s\big|_{t=-R^2}\,dX.
\end{align*}
The Cauchy-Schwartz inequality implies
\begin{align*}
I'(R)&H(R)-I(R)H'(R) \\
&=4R^3\bigg(\int_{\R^{n+1}_+}y^a\Big(\partial_tU+\nabla U\frac{X}{2t}\Big)^2\G_s\,dX\bigg)
\bigg(\int_{\R^{n+1}_+}y^aU^2\G_s\,dX\bigg)\big|_{t=-R^2} \\
&\quad-4R^3\bigg(\int_{\R^{n+1}_+}y^aU\Big(\partial_tU+\nabla U\frac{X}{2t}\Big)\G_s\,dX\bigg)^2\big|_{t=-R^2}\geq0.
\end{align*}
Hence $N'(R)=\frac{I'(R)H(R)-I(R)H'(R)}{[H(R)]^2}\geq0$ and $N(R)$ is monotone
nondecreasing in $R\in(0,1)$. Moreover, $N(R)$ is constant if and only if equality is achieved in
the Cauchy--Schwartz inequality. Equivalently, if and only if there is $\kappa\in\R$ such that
$\kappa U=2t\partial_tU+X\nabla U$,
that is, if and only if $U$ is parabolically homogeneous of degree $\kappa$.}
\end{proof}

\section{H\"older and Schauder estimates}\label{Section:Holder}

In this last section we first present our new characterizations of the parabolic H\"older and Zygmund spaces,
Theorems \ref{caract} and \ref{prop:2}. We then prove Theorems \ref{thm:Holder},
\ref{thm:solution Poisson problem}, \ref{thm:Schauder} and \ref{thm:interiorSchauder}.

\subsection{The novel characterization of parabolic H\"older spaces}

 Parallel to the elliptic case, see \cite[Chapter V]{Stein},
the subordinated Poisson semigroup, $P_y\equiv e^{-yH^{1/2}}$, $y>0$, can be defined by
\begin{align*}
P_y u(t,x) &:= \frac{y}{2\sqrt{\pi}}\int_0^\infty e^{-y^2/(4\tau)}e^{-\tau H}u(t,x)\,\frac{d\tau}{\tau^{3/2}} \\
&=\frac{y}{ 2 \sqrt{\pi}} \int_0^\infty\int_{\R^n}
\frac{e^{-(y^2+|z|^2)/(4\tau)}}{(4\pi \tau)^{n/2}}  u(t-\tau,x-z)\,dz\,\frac{d\tau}{\tau^{3/2}}.
\end{align*} 
This corresponds to the case $s=1/2$ in our extension problem (see Theorem \ref{thm:extension}).
It can be easily seen that $P_{y_1}(P_{y_2}u)=P_{y_1+y_2}u$, $y_1,y_2\geq0$.

\begin{thm}[Parabolic H\"older spaces and growth of $P_yu(t,x)$]\label{caract}
Let $0<\alpha<1$. Suppose that $u=u(t,x)$ is a bounded function on $\R^{n+1}$.
Then $u$ is in the parabolic H\"older space $C^{\alpha/2,\alpha}_{t,x}(\mathbb{R}^{n+1})$ if and only if
there exists a constant $C$ such that
$$\bigg\|\frac{\partial}{\partial y}P_yu(t,x)\bigg\|_{L^\infty(\R^{n+1})}\le Cy^{-1+\alpha},$$
for all $y>0$.
Moreover, if $C_1$ is the infimum of the constants $C$ above then $C_1\sim[u]_{C^{\alpha/2,\alpha}_{t,x}}$.
\end{thm}

Next, for $\alpha>0$ we let $k=[\alpha]+1$, the smallest integer bigger that $\alpha$. 
Motivated by Theorem \ref{caract} we define the space of functions
$$ \Lambda^\alpha\equiv\Lambda^{\alpha/2,\alpha}_{t,x}:=
\bigg\{u\in L^\infty(\R^{n+1}):\bigg\|\frac{\partial^k}{\partial y^k}P_yu(t,x)\bigg\|_{L^\infty(\R^{n+1})}\le Cy^{-k+\alpha},
~\hbox{for}~y>0,~\hbox{for some}~C\bigg\},$$
under the norm $\|u\|_{\Lambda^\alpha}:=\|u\|_{L^\infty(\R^{n+1})}+C_1$, where $C_1$ is the infimum of the constants $C$
appearing in the definition above.

\begin{thm}\label{prop:2}
Let $\alpha>0$ and $u\in L^\infty(\R^{n+1})$.
\begin{enumerate}[$(1)$]
\item Suppose that $0<\alpha< 2$. Then $u\in\Lambda^\alpha$ if and only if there exists a constant $C>0$ such that
\begin{equation}\label{Lamda}
|u(t-\tau,x-z)+u(t-\tau,x+z)-2u(t,x)|\le C(\tau^{1/2}+|z|)^\alpha,
\end{equation}
for all $(t,x)\in\R^{n+1}$. In this case, if $C_2$ denotes the least constant $C$ for which the inequality
above is true, then $\|u\|_{L^\infty(\R^{n+1})}+C_2\sim\|u\|_{\Lambda^\alpha}$.
\item Suppose that $0<\alpha<2$. Then $u\in \Lambda^\alpha$ if and only
if $u(t,\cdot) \in \Lambda^\alpha(\mathbb{R}^n)$ and $u(\cdot,x)\in\Lambda^{\alpha/2}(\mathbb{R})$
uniformly on $x$ and $t$, where the spaces $\Lambda^\alpha(\R^n)$
and $\Lambda^{\alpha/2}(\R)$ are the classical Zygmund spaces
defined in the usual way, see \cite[p.~147]{Stein}.
\item Suppose that $\alpha>2$. Then $u\in\Lambda^\alpha$ if and only if
$$\frac{\partial^2}{\partial x_i^2}P_yu(t,x)\in\Lambda^{\alpha-2},~\hbox{for}~i=1,\ldots,n,\quad
\hbox{and}\quad\frac{\partial}{\partial t}P_yu(t,x)\in\Lambda^{\alpha-2}.$$
In this case we have the equivalence
$$\|u\|_{\Lambda^\alpha}\sim 
\|u\|_{L^\infty(\R^{n+1})}+\sum_{i=1}^n\bigg\| \frac{\partial^2}{\partial x_i^2}P_yu(t,x)\bigg\|_{\Lambda^{\alpha-2}}+ 
\bigg\|\frac{\partial}{\partial t}P_yu(t,x)\bigg\|_{\Lambda^{\alpha-2}}.$$
\item In particular, we have the following equivalences.
\begin{enumerate}[$(i)$]
\item For $0<\alpha<1$, $\Lambda^\alpha=C^{\alpha/2,\alpha}_{t,x}(\R^{n+1})$.
\item For $1<\alpha<2$, $\Lambda^\alpha=C^{\alpha/2,1+(\alpha-1)}_{t,x}(\R^{n+1})$.
\item For $2<\alpha<3$, $\Lambda^\alpha=C^{1+(\alpha/2-1),2+(\alpha-2)}_{t,x}(\R^{n+1})$.
\item For $k<\alpha<k+1$, $\Lambda^\alpha=C^{[k/2]+(\alpha/2-[k/2]),k+(\alpha-k)}_{t,x}(\R^{n+1})$.
\item When $\alpha=1$, $\Lambda^1=\Lambda^{1/2,1}_{t,x}$, which is the class of functions
that are $1/2$-H\"older in time and $1$-Zygmund in space, see \eqref{Lamda}. This space contains
the space $C^{1/2,1}_{t,x}(\R^{n+1})$ of $1/2$-H\"older continuous functions in $t$ and
Lipschitz continuous functions in $x$.
\item When $\alpha=2$, $\Lambda^2=\Lambda^{1,2}_{t,x}$, which is the class of functions $u$
that are Lipschitz in time and have gradient $\nabla_xu$ in the $1$-Zygmund space.
\end{enumerate}
\end{enumerate}
\end{thm}

For the convenience of the reader, we will next prove Theorems \ref{thm:Holder},
\ref{thm:solution Poisson problem}, \ref{thm:Schauder} and \ref{thm:interiorSchauder},
and postpone the proofs of Theorems \ref{caract} and \ref{prop:2} until Subsection \ref{theproofs}.

\subsection{H\"older estimates}

It can be checked by using the Cauchy Integral Theorem that for $0<s<1$, $\rho\in\R$ and $\xi\in\R^n$,
$$(i\rho+|\xi|^2)^s=\frac{1}{c_s}\int_0^\infty\big(e^{-\tau(i\rho+|\xi|^2)^{1/2}}-1\big)^{[2s]+1}\,\frac{d\tau}{\tau^{1+2s}},$$
where $[2s]$ denotes the integer part of $0<2s<2$, and
$$c_s=\int_0^\infty(e^{-\tau}-1)^{[2s]+1}\,\frac{d\tau}{\tau^{1+2s}}.$$
With this we write the following equivalent formula:
\begin{equation}\label{1}
H^su(t,x)=\frac{1}{c_s}\int_0^\infty\big(e^{-\tau H^{1/2}}-\Id\big)^{[2s]+1}u(t,x)\,\frac{d\tau}{\tau^{1+2s}}.
\end{equation}

\begin{proof}[Proof of Theorem \ref{thm:Holder}]
In view of Theorems \ref{caract} and \ref{prop:2},
it is enough to show that if $u\in\Lambda^\alpha$ and $2s<\alpha$ then $H^{s}u \in \Lambda^{\alpha-2s}$ and
there exists a constant $C$ such that
$$\|H^su\|_{\Lambda^{\alpha-2s}}\leq C\|u\|_{\Lambda^\alpha}.$$

Assume first that $0<s<1/2$, so $[2s]=0$. 
For any $\alpha <1$ we have $P_yu\rightarrow u$ uniformly. Indeed,
\begin{equation}\label{uniform} 
\begin{aligned}
|P_yu(t,x)-u(t,x)|&=
\bigg|\frac{y}{2\sqrt{\pi}}\int_0^\infty e^{-y^2/(4\tau)}\big(e^{-\tau H}u(t,x)-u(t,x)\big)\,\frac{d\tau}{\tau^{3/2}}\bigg|\\
&\le Cy\int_0^\infty e^{-y^2/(4\tau)}\int_{\R^{n}}
\frac{e^{-|z|^2/(4\tau)}}{\tau^{n/2}} |u(t-\tau,x-z)- u(t,x)|\,dz\,\frac{d\tau}{\tau^{3/2}} \\
&\le C\|u\|_{\Lambda^\alpha} y^\alpha,\qquad\hbox{for every}~(t,x)\in\R^{n+1}.
\end{aligned}
\end{equation}
Then
\begin{align*}
\bigg|\frac{\partial}{\partial y}P_y(H^{s}u)(t,x)\bigg|
&= \bigg|\frac{\partial}{\partial y}P_y\bigg(\frac{1}{c_s}\int_0^\infty(P_\tau u(t,x)-u(t,x))\,\frac{d\tau}{\tau^{1+2s}}\bigg)\bigg| \\ 
&\leq C_s\bigg\{\int_0^y+\int_y^\infty\bigg\}\bigg|\frac{\partial}{\partial y}P_y(P_\tau u(t,x)- u(t,x))\bigg|\,\frac{d\tau}{\tau^{1+2s}}.
\end{align*}
By using the semigroup property we can estimate the first integral above as follows: \begin{align*}
\int_0^y\bigg|\frac{\partial}{\partial y}P_y\int_0^\tau\frac{\partial}{\partial r}P_ru(t,x)\,dr\bigg|\,\frac{d\tau}{\tau^{1+2s}} 
&\le  \int_0^y\int_0^\tau\bigg|\frac{\partial^2}{\partial\nu^2}P_\nu u(t,x)\Big|_{\nu=y+r}\bigg|\,dr\,\frac{d\tau}{\tau^{1+2s}} \\
&\le C\|u\|_{\Lambda^\alpha} \int_0^y\int_0^\tau(y+r)^{-2+\alpha}\,dr\,\frac{d\tau}{\tau^{1+2s}} \\
&= C\|u\|_{\Lambda^\alpha}y^{-1+\alpha}\int_0^y\int_0^{\tau/y}(1+r)^{-2+\alpha}\,dr\,\frac{d\tau}{\tau^{1+2s}} \\
&\le C\|u\|_{\Lambda^\alpha}y^{-1+\alpha}
\int_0^y\bigg(\frac{\tau}{y}\bigg)\frac{d\tau}{\tau^{1+2s}}=C\|u\|_{\Lambda^\alpha}y^{-1+\alpha-2s}.
\end{align*}
Observe that the second integral is controlled by 
\begin{align*}
\int_y^\infty&\bigg|\frac{\partial}{\partial y}P_yP_\tau u(t,x)\bigg|\,\frac{d\tau}{\tau^{1+2s}}+
\int_y^\infty\bigg|\frac{\partial}{\partial y}P_yu(t,x)\bigg|\,\frac{d\tau}{\tau^{1+2s}} \\
&\le C\|u\|_{\Lambda^\alpha}\int_y^\infty(y+\tau)^{-1+\alpha}\,\frac{d\tau}{\tau^{1+2s}}
+C\|u\|_{\Lambda^\alpha}\int_y^\infty y^{-1+\alpha}\frac{d\tau}{\tau^{1+2s}}
\le C\|u\|_{\Lambda^\alpha}y^{-1+\alpha-2s}.
\end{align*}

Let us now assume that $1/2\leq s<1$, so $[2s]=1$. Then
\begin{equation}\label{aestimar}
\begin{aligned}
\bigg|\frac{\partial}{\partial y}P_y(H^{s}u)(t,x)\bigg|
&= \bigg|\frac{\partial}{\partial y}P_y\bigg(\frac{1}{c_s}\int_0^\infty(P_{2\tau} u(t,x)-2P_\tau u(t,x)+u(t,x))\,\frac{d\tau}{\tau^{1+2s}}\bigg)\bigg| \\ 
&\leq C_s\bigg|\bigg\{\int_0^y+\int_y^\infty\bigg\}\frac{\partial}{\partial y}P_y(P_{2\tau} u(t,x)- P_{\tau} u(t,x)) \,\frac{d\tau}{\tau^{1+2s}}\bigg| \\ &\quad  +C_s \bigg|\bigg\{\int_0^y+
\int_y^\infty\bigg\} \frac{\partial}{\partial y}P_y(P_{\tau} u(t,x)-u(x,t))\,\frac{d\tau}{\tau^{1+2s}}\bigg|.
\end{aligned}
\end{equation}
By using the semigroup property we can estimate the first integral in the first term above as follows:
\begin{align*}
\int_0^y\bigg|\frac{\partial}{\partial y}P_y\int_\tau^{2\tau}\frac{\partial}{\partial r}P_ru(t,x)\,dr\bigg|\,\frac{d\tau}{\tau^{1+2s}} 
&\le  \int_0^y\int_\tau^{2\tau}\bigg|\frac{\partial^2}{\partial\nu^2}P_\nu u(t,x)\Big|_{\nu=y+r}\bigg|\,dr\,\frac{d\tau}{\tau^{1+2s}} \\
&\le C\|u\|_{\Lambda^\alpha} \int_0^y\int_\tau^{2\tau}(y+r)^{-2+\alpha}\,dr\,\frac{d\tau}{\tau^{1+2s}} \\
&= C\|u\|_{\Lambda^\alpha}y^{-1+\alpha}\int_0^y\int_{\tau/y}^{2\tau/y}(1+r)^{-2+\alpha}\,dr\,\frac{d\tau}{\tau^{1+2s}} \\
&\le C\|u\|_{\Lambda^\alpha}y^{-1+\alpha}
\int_0^y\bigg(\frac{\tau}{y}\bigg)\frac{d\tau}{\tau^{1+2s}}=C\|u\|_{\Lambda^\alpha}y^{-1+\alpha-2s}.
\end{align*}
Observe that the second integral in the first term above is controlled by 
\begin{align*}
\int_y^\infty&\bigg|\frac{\partial}{\partial y}P_yP_{2\tau }u(t,x)\bigg|\,\frac{d\tau}{\tau^{1+2s}}+
\int_y^\infty\bigg|\frac{\partial}{\partial y}P_yP_{\tau }u(t,x)\bigg|\,\frac{d\tau}{\tau^{1+2s}} \\
&\le C\|u\|_{\Lambda^\alpha}\int_y^\infty(y+2\tau)^{-1+\alpha}\,\frac{d\tau}{\tau^{1+2s}}
+C\|u\|_{\Lambda^\alpha}\int_y^\infty (y+\tau)^{-1+\alpha}\frac{d\tau}{\tau^{1+2s}}
\le C\|u\|_{\Lambda^\alpha}y^{-1+\alpha-2s}.
\end{align*}
For the second term in \eqref{aestimar} we proceed as in the case $[2s]=0$.
\end{proof}

\subsection{The Poisson problem}

We present the proof of Theorem \ref{thm:solution Poisson problem}.

\begin{proof}[Proof of Theorem \ref{thm:solution Poisson problem}]
The first identity for $u$ in the statement follows by multiplying 
the following formula with the Gamma function
$$(i\rho+|\xi|^2)^{-s} = \frac1{\Gamma(s)} \int_0^\infty e^{-\tau(i\rho+|\xi|^2)}\,\frac{d\tau}{\tau^{1-s}},$$
by $\widehat{f}(\rho,\xi)$ and taking inverse Fourier transform. The second one is obtained by writing
$e^{-\tau H}f(t,x)$ with its kernel, see \eqref{eq:el calor}. Finally, estimates $(1)$ and $(2)$ are obvious.
\end{proof}

\subsection{Schauder estimates}

For any $s>0$ we have
$$(i\rho+|\xi|^2)^{-s}=\frac{1}{\Gamma(2s)}\int_0^\infty e^{-\tau(i\rho+|\xi|^2)^{1/2}}\,\frac{d\tau}{\tau^{1-2s}},$$
so the $L^2$ solution $u$ to $H^su=f$ can also be written as
\begin{equation}\label{2}
u(t,x)=H^{-s}f(t,x)=\frac{1}{\Gamma(2s)}\int_0^\infty e^{-\tau H^{1/2}}f(t,x)\,\frac{d\tau}{\tau^{1-2s}}.
\end{equation}

\begin{proof}[Proof of Theorem \ref{thm:Schauder}]
For the first part, it is enough to show that
if $u$ solves $H^su=f$ in $\R^{n+1}$ in the sense that $u=H^{-s}f$, and $f\in\Lambda^\alpha$, then
$u\in\Lambda^{\alpha+2s}$ and there exists a constant $C$ such that
$$\|u\|_{\Lambda^{\alpha+2s}}\leq C\big(\|f\|_{\Lambda^\alpha}+\|u\|_{L^\infty(\R^{n+1})}\big).$$
Indeed, we observe that 
\begin{align*}
\bigg|\frac{\partial}{\partial y}P_yu(t,x)\bigg|
&\leq\frac{1}{\Gamma(2s)}\int_0^\infty\bigg|\frac{\partial}{\partial y}P_{y+\tau}f(t,x)\bigg|\,\frac{d\tau}{\tau^{1-2s}} \\
&\le C\|f\|_{\Lambda^{\alpha}}\int_0^\infty(y+\tau)^{-1+\alpha}\,\frac{d\tau}{\tau^{1-2s}}
=C_{s,\alpha}\|f\|_{\Lambda^{\alpha}}y^{(\alpha+2s)-1}.
\end{align*}

To prove the second part of the statement, it is enough to show that
$$\|u\|_{\Lambda^{s,2s}_{t,x}}\leq C\big(\|f\|_{L^\infty(\R^{n+1})}+\|u\|_{L^\infty(\R^{n+1})}\big).$$
Suppose that $2s<1$. From Lemma \ref{lema1} (ii) we see that for $f$ bounded we have
$$\bigg|y\frac{\partial}{\partial y}P_yf(t,x)\bigg|\le C\|f\|_{L^\infty(\R^{n+1})}.$$
Hence
\begin{align*}
\bigg|y\frac{\partial}{\partial y}P_yu(t,x)\bigg|
&\leq  C\bigg|y\int_0^\infty\frac{\partial}{\partial y}P_{y+\tau}f(t,x)\,\frac{d\tau}{\tau^{1-2s}}\bigg| \\
&\leq  C\|f\|_{L^\infty(\R^{n+1})}y\int_0^\infty\frac{\tau^{2s-1}}{(y+\tau)}\,d\tau \\
&= C \|f\|_{L^\infty(\R^{n+1})}y^{2s}\int_0^\infty\frac{r^{2s-1}}{(1+r)}\,dr=C \|f\|_{L^\infty(\R^{n+1})}y^{2s}.
\end{align*}
In the case when $2s\geq1$ we proceed in a similar way but using now that
$$\bigg|y^2\frac{\partial^2}{\partial y^2}P_yf(t,x)\bigg|\le C\|f\|_{L^\infty(\R^{n+1})},$$
which follows from \eqref{3Poisson}.
\end{proof}

\begin{proof}[Proof of Theorem \ref{thm:interiorSchauder}]
 {Let $\eta=\eta(t,x)\in C^\infty_c((0,1)\times B_{2})$
such that $\eta=1$ in $(1/8,7/8)\times B_{3/2}$, $0\leq\eta\leq1$ in $\R^{n+1}$. Let (see Theorem \ref{thm:solution Poisson problem})
$$w(t,x)=H^{-s}(\eta f)(t,x)=\int_{-\infty}^t\int_{\R^n}K_{-s}(t-\tau,x-z)(\eta f)(\tau,z)\,dz\,d\tau.$$
We estimate the $L^\infty $norm of $w$ first. Since the kernel $K_{-s}$ is positive and
$\eta\geq0$ is a smooth function with compact support in $(0,1)\times B_2$,
\begin{equation}\label{prima}
\|w\|_{L^\infty(\R^{n+1})} \leq \|f\|_{L^\infty((0,1)\times B_2)}\|H^{-s}\eta\|_{L^\infty(\R^{n+1})}
=C\|f\|_{L^\infty((0,1)\times B_2)}.
\end{equation}
Next we estimate the H\"older--Zygmund norms of $w$ depending on the assumptions on $f$.}

 {If $f$ satisfies $(1)$ then $\eta f\in\Lambda^\alpha$.
It is easy to check that
\begin{equation}\label{seconda}
\|\eta f\|_{\Lambda^\alpha}\leq\|f\|_{L^\infty((0,1)\times B_2)}\|\eta\|_{\Lambda^\alpha}
+\|f\|_{\Lambda^\alpha((0,1)\times B_2)}\|\eta\|_{L^\infty(\R^{n+1})}
\leq C\|f\|_{\Lambda^\alpha((0,1)\times B_2)},
\end{equation}
where by $\Lambda^\alpha(D)$, $D\subset\R^{n+1}$, we mean
the corresponding H\"older--Zygmund space
given by the characterizations of Theorem \ref{prop:2} $(4)$
with $D$ in place of $\R^{n+1}$.
From Theorem \ref{thm:Schauder} $(1)$, $w\in\Lambda^{\alpha+2s}$
with the estimate
$$\|w\|_{\Lambda^{\alpha+2s}} \leq C\big(\|\eta f\|_{\Lambda^\alpha}+\|w\|_{L^\infty(\R^{n+1})}\big)
\leq C\|f\|_{\Lambda^\alpha((0,1)\times B_2)},$$
where we applied \eqref{prima} and \eqref{seconda}.}

 {In a similar way, if $f$ satisfies $(2)$ then $\eta f\in L^\infty(\R^{n+1})$ and
we can apply Theorem \ref{thm:Schauder} $(2)$ to conclude
that $w\in\Lambda^{2s}$ with the estimate
$$\|w\|_{\Lambda^{2s}}\leq C\big(\|\eta f\|_{L^\infty(\R^{n+1})}+\|w\|_{L^\infty(\R^{n+1})}\big)
\leq C\|f\|_{L^\infty((0,1)\times B_2)}.$$}

 {Let now $u$ be as in the statement. Then
$$H^sw=f=H^su\quad\hbox{in}~(1/8,7/8)\times B_{3/2}.$$
Hence $H^s(u-w)=0$ in $(1/8,7/8)\times B_{3/2}$. By the interior
derivative estimates in Corollary \ref{cor:interior} (2),
$u-w$ is smooth in $(1/4,3/4)\times B_1$ and all its time and space derivatives,
and, in particular, all the parabolic H\"older and Zygmund norms of $u-w$ in $(1/4,3/4)\times B_1$
are bounded by
$$C\|u-w\|_{L^\infty((-\infty,1)\times\R^n)}\leq C\big(\|u\|_{L^\infty((-\infty,1)\times\R^n)}
+\|f\|_{L^\infty((0,1)\times B_2)}\big).$$
The conclusion follows by noticing, in case $(1)$, that
\begin{align*}
\|u\|_{\Lambda^{\alpha+2s}((1/4,3/4)\times B_1)}
&\leq\|u-w\|_{\Lambda^{\alpha+2s}((1/4,3/4)\times B_1))}+\|w\|_{\Lambda^{\alpha+2s}} \\
&\leq C\big(\|u\|_{L^\infty((-\infty,1)\times\R^n)}+\|f\|_{\Lambda^\alpha((0,1)\times B_2)}\big),
\end{align*}
and, in case $(2)$,
\begin{align*}
\|u\|_{\Lambda^{2s}((1/4,3/4)\times B_1)}
&\leq\|u-w\|_{\Lambda^{2s}((1/4,3/4)\times B_1))}+\|w\|_{\Lambda^{2s}} \\
&\leq C\big(\|u\|_{L^\infty((-\infty,1)\times\R^n)}+\|f\|_{L^\infty((0,1)\times B_2)}\big).
\end{align*}}
\end{proof}

\subsection{Proofs of Theorems \ref{caract} and \ref{prop:2}}\label{theproofs}

To prove Theorem \ref{caract} we need two preliminary lemmas.
Observe that $P_yu(t,x)$ can be written as a convolution of $u$ over $\R^{n+1}$ with the kernel
\begin{equation}\label{defPoisson}
P_y(\tau,z) = \frac{y}{2\sqrt{\pi}} 
\frac{e^{-(y^2+|z|^2)/(4\tau)}}{(4\pi \tau)^{n/2}\tau^{3/2}}\chi_{(0,\infty)}(\tau),\quad \tau\in\R,~z\in\R^n,~y>0.
\end{equation}

\begin{lem}\label{lema1}
Let $P_y(\tau,z)$ be the kernel given by \eqref{defPoisson}. We have 
\begin{itemize}
\item [(i)]$\displaystyle \int_{\mathbb{R}^{n+1}} P_y(\tau,z)\,dz\,d\tau=1$;
\item [(ii)] $\displaystyle \int_{\mathbb{R}^{n+1}} |\partial_y P_y(\tau,z)|\,dz\,d\tau
\le \frac {C}{y}$, and $\displaystyle \int_{\mathbb{R}^{n+1}} |\partial_{z_i} P_y(\tau,z)|\,dz\,d\tau
\le \frac {C}{y}$, $i=1,\dots, n$;
\item [(iii)] $\displaystyle \int_{\mathbb{R}^{n+1}} 
\partial_y P_y(\tau,z)\,dz\,d\tau=0$.
\end{itemize}
\end{lem}

\begin{proof}
(i) is  easy to check. 
For (ii),  we observe that 
\begin{equation}\label{formula2}
|\partial_yP_y(\tau,z)|\le C\frac{e^{-(y^2+|z|^2)/(c\tau)}}{\tau^{(n+3)/2}} \chi_{(0,\infty)}(\tau).
\end{equation}  
Hence 
$$\int_{\mathbb{R}^{n+1}} |\partial_y P_y(\tau,z)|\,dz\,d\tau\le
 C \int_{0}^\infty \frac{ e^{-y^2/(c\tau)}}{\tau^{3/2}}\,d\tau=\frac{C}{y}.$$
The derivatives with respect to the spatial variables $z_i$ can be handled in a parallel way.
(iii) is a consequence of (i) and (ii).
\end{proof}

\begin{lem} \label{lema2}
Let $u=u(t,x)\in L^\infty(\mathbb{R}^{n+1})$ and $0<\alpha <1$. The following conditions are equivalent.
\begin{itemize}
\item[(1)] There exists a constant $A$ such that
$$\|\partial_{y} P_yu\|_{L^\infty(\R^{n+1})}\le Ay^{-1+\alpha}.$$
\item[(2)]There exist constants $B_i$, $i=1,\ldots,n$ and $C_1$ such that
$$\|\partial_{x_i} P_yu\|_{L^\infty(\R^{n+1})}\le B_iy^{-1+\alpha},~i=1,\ldots,n,\quad\hbox{and}\quad
\| \partial_{t}P_yu\|_{L^\infty(\R^{n+1})}\le C_1y^{-2+\alpha}.$$
\end{itemize}
Moreover the infimum of the constants $A$ appearing in {\rm (1)}
is equivalent to the infimum of the constants $B_i$ and $C$ appearing in {\rm (2)}.
\end{lem}

\begin{proof}
Let us prove that (1) implies (2).
The converse can be done with a parallel argument, so the details in that case are left to the interested reader.
By using the semigroup property we have
$$\frac{\partial^2 P_yu(t,x)}{\partial y\,\partial x_i} = 
\bigg(\frac{\partial P_{y/2}}{\partial x_i}\bigg)\ast\bigg(\frac{\partial P_yu(t,x)}{\partial y}\bigg)\bigg|_{y/2},$$
\begin{equation}\label{semigrupo3}
\frac{\partial^2 P_yu(t,x)}{\partial x^2_i} = 
\bigg(\frac{\partial P_{y/2}}{\partial x_i}\bigg)\ast\bigg(\frac{\partial P_yu(t,x)}{\partial x_i }\bigg)\bigg|_{y/2},
\end{equation}
and
$$\frac{\partial^2 P_yu(t,x)}{ \partial y^2} = 
\bigg(\frac{\partial P_{y/2}}{ \partial y}\bigg)\ast\bigg(\frac{\partial P_yu(t,x)}{\partial y}\bigg)\bigg|_{y/2}.$$
Hence if we assume that $\|\partial_yP_yu\|_{L^\infty(\R^{n+1})}\le Ay^{-1+\alpha}$ then
by (ii) in Lemma \ref{lema1},
\begin{equation}\label{derseg}
\bigg\|\,\frac{\partial^2 P_yu}{\partial y^2}\bigg\|_{L^\infty(\R^{n+1})}\le Cy^{-2+\alpha},
\quad\hbox{and}\quad
\bigg\|\frac{\partial^2 P_yu}{\partial y\,\partial x_i}\bigg\|_{L^\infty(\R^{n+1})}\le Cy^{-2+\alpha}.
\end{equation}
Moreover,
$$\|\partial_{x_i}P_yu(t,x)\|_{L^\infty(\R^{n+1})}\le 
\|\partial_{x_i}P_y\|_{L^1(\R^{n+1})}\|u\|_{L^\infty(\R^{n+1})} \le 
\frac{C}{y}\|u\|_{L^\infty(\R^{n+1})}.$$
Then  $\displaystyle \partial_{x_i}P_yu(t,x)  \rightarrow 0$, as $y\rightarrow \infty$.
Therefore
$$\partial_{x_i}P_yu(t,x)=-\int_y^\infty \frac{\partial^2 P_{y'} u(t,x)}{\partial y\,\partial x_i} dy'.$$
By using the second estimate in \eqref{derseg} we get 
$$\|\partial_{x_i}P_yu\|_{L^\infty(\R^{n+1})}\le Cy^{-1+\alpha}.$$
Analogously, by \eqref{semigrupo3} we can get
$$\bigg\|\frac{\partial^2}{\partial x^2_i}P_yu\bigg\|_{L^\infty(\R^{n+1})}\le Cy^{-2+\alpha}.$$
As $P_yu(t,x)$ satisfies the equation
$$\partial^2_yP_yu(t,x) = \partial_t P_yu(t,x)-\sum_{i=1}^n\frac{\partial^2}{\partial x_i^2}P_yu(t,x),$$
we readily see from \eqref{derseg} and the previous estimate that $\|\partial_tP_yu\|_{L^\infty(\R^{n+1})}\le Cy^{-2+\alpha}$.
\end{proof}

\begin{proof}[Proof of Theorem \ref{caract}]
Assume that $u\in C_{t,x}^{\alpha/2,\alpha}(\R^{n+1})$. By using (i) in Lemma \ref{lema1} we can write 
$$\partial_yP_yu(t,x)= 
\int_{\mathbb{R}^{n+1}}\partial_yP_y(\tau,z)\big(u(t-\tau,x-z)-u(t,x)\big)\,dz\,d\tau,$$
for every $(t,x)\in\R^{n+1}$. Therefore, by using \eqref{formula2} and the regularity of $u$,
\begin{align*}
\|\partial_yP_yu(t,x)\|_{L^\infty(\R^{n+1})} &\le C\|u\|_{C^{\alpha/2,\alpha}_{t,x}}
\int_0^\infty\int_{\R^n}\frac{e^{-(y^2+|z|^2)/(c\tau)}}{\tau^{(n+3)/2}}(\tau^{1/2}+|z|)^\alpha\,dz\,d\tau \\ 
 &= C\|u\|_{C^{\alpha/2,\alpha}_{t,x}}y^{-1+\alpha}.
\end{align*}
Hence $u\in\Lambda^\alpha$. For the converse, suppose that $u\in\Lambda^\alpha$ and write
\begin{multline*}
u(t+\tau,x+z)-u(t,x) \\
=\big(P_yu(t+\tau,x+z)-P_yu(t,x)\big)+
\big(u(t+\tau,x+z)-P_yu(t+\tau,x+z)\big)+\big(P_yu(t,x)-u(t,x)\big).
\end{multline*}
Next we set $y=|\tau|^{1/2}+|z|$. Then, by the Mean Value Theorem and Lemma \ref{lema2}, for some $0<\theta<1$,
\begin{align*}
|P_yu(t+\tau,x+z)-P_yu(t,x)| &\le |\partial_tP_yu(t+\theta\tau,x+\theta z)||\tau|+|\nabla_xP_yu(t+\theta\tau,x+\theta z)||z| \\
&\le  C(|\tau|^{1/2}+|z|)^{-2+\alpha}|\tau|+C(|\tau|^{1/2}+|z|)^{-1+\alpha}|z| \\
&\le C (|\tau|^{1/2}+|z|)^{\alpha}. 
\end{align*}  
On the other hand,   
\begin{align*}
|u(t+\tau,x+z)-P_yu(t+\tau,x+z)| &= \bigg|-\int_0^y \frac{\partial P_{y'}u(t+\tau,x+z)}{\partial y'}\,dy'\bigg| \\
&\le  \int_0^y\bigg\|\frac{\partial P_{y'}u(t+\tau,x+z)}{\partial y'}\bigg\|_{L^\infty(\R^{n+1})}\,dy' \\
&\le  C\int_0^y{y'}^{-1+\alpha}\,dy'=Cy^\alpha=C(\tau^{1/2}+|z|)^\alpha.
\end{align*}
A similar estimate can be made in the third summand above. We conclude that $u\in C^{\alpha/2,\alpha}_{t,x}(\R^{n+1})$.
\end{proof}
 
In order to prove Theorem \ref{prop:2} we need a preliminary result.

\begin{lem}
Let $u\in L^\infty(\mathbb{R}^{n+1})$ and $\alpha>0$. Let $k$ and $l$ 
be two integers greater than $\alpha$. Then the two conditions
$$\bigg\|\frac{\partial^kP_yu(t,x)}{\partial y^k}\bigg\|_{L^\infty(\R^{n+1})}\le A_ky^{-k +\alpha},\quad\hbox{and}\quad
\bigg\|\frac{\partial^lP_yu(t,x)}{\partial y^l}\bigg\|_{L^\infty(\R^{n+1})}\le A_ly^{-l +\alpha},$$
for $y>0$, where $A_k$ and $A_l$ are some positive constants, are equivalent.
\end{lem}

\begin{proof}
The reasoning of Lemma \ref{lema2} drives to the fact that the bound condition for
the integer $k$ implies the condition for $l=k+1$.
An integration device gives the converse.
\end{proof}

\begin{proof}[Proof of Theorem \ref{prop:2} (1)]
We first obtain some estimates for the second derivatives of the kernel $P_y(\tau,z)$.
It is not difficult to check that for some constant $c>0$,
\begin{equation}\label{segunPoisson}
\bigg|\frac{\partial^2P_y(\tau,z)}{\partial y^2}\bigg|\le C\frac{e^{-(y^2+|z|^2)/(c\tau)}}{\tau^{(n+4)/2}}\chi_{(0,\infty)}(\tau),
\end{equation}
for every $(\tau,z)\in\R^{n+1}$, $y>0$. Therefore
\begin{equation}\label{3Poisson}
\int_{\R^{n+1}} \bigg|\frac{\partial^2P_y(\tau,z)}{\partial y^2}\bigg|\,dz\,d\tau\le \frac{C}{y^2}.
\end{equation}
Then, by using (ii) in Lemma \ref{lema1}, we get  
\begin{equation}\label{4Poisson}
\int_{\R^{n+1}} \frac{\partial^2P_y(\tau,z)}{\partial y^2}\,dz\,d\tau=0,\quad\hbox{for every}~y>0.
\end{equation}
Also, from \eqref{segunPoisson} we have 
\begin{equation}\label{6Poisson}
\bigg|\frac{\partial^2P_y(\tau,z)}{\partial y^2}\bigg|\le 
C\frac{e^{-y^2/(c\tau)}}{\tau^{(n+4)/2}}\le\frac{C}{y^{n+4}},
\end{equation}
and 
\begin{equation}\label{5Poisson}
\bigg|\frac{\partial^2P_y(\tau,z)}{\partial y^2}\bigg| \le 
C\frac{e^{-|z|^2/(c\tau)}}{\tau^{(n+4)/2}}\le\frac{C}{(\tau^{1/2}+|z|)^{n+4}},
\end{equation}
for $\tau>0$. We finally notice that
\begin{equation}\label{simetria}
\frac{\partial^2P_y(\tau,z)}{\partial y^2}=\frac{\partial^2P_y(\tau,-z)}{\partial y^2}.
\end{equation}

Now we are in position to prove part (1) in Theorem \ref{prop:2}. Suppose first that $u$
satisfies condition \eqref{Lamda} with $0<\alpha <2$ on the incremental quotients. Then from \eqref{4Poisson} and \eqref{simetria}
it is possible to write
$$\frac{\partial^2P_yu(t,x)}{\partial y^2}= \frac12\int_{\mathbb{R}^{n+1}} 
\frac{\partial^2P_y(\tau,z)}{\partial y^2}\big(u(t-\tau,x+z)+u(t-\tau,x-z)-2u(t,x)\big)\,dz\,d\tau.$$
Hence, by using \eqref{6Poisson}, \eqref{5Poisson} and the hypothesis on $u$, 
\begin{align*}
\bigg\|\frac{\partial^2P_yu}{\partial y^2}\bigg\|_{L^\infty(\R^{n+1})}
&\le \frac{C}{y^{n+4}}\int_{\tau^{1/2}+|z|<y}(\tau^{1/2}+|z|)^\alpha\,dz\,d\tau 
+ C\int_{\tau^{1/2}+|z|\geq y}(\tau^{1/2}+|z|)^{\alpha-n-4}\,dz\,d\tau \\
&\le Cy^{-2+\alpha}
 \end{align*}
The last estimate for the first integral above follows by using the change of variables $\rho=\tau^{1/2}$:
\begin{align*}
\int_{\tau^{1/2}+|z|<y}(\tau^{1/2}+|z|)^\alpha\,dz\,d\tau &= 
 2\int_{\rho+|z|<y}\rho(\rho+|z|)^\alpha\,dz\,d\rho \\
 &\le 2\int_{\rho+|z|<y}(\rho+|z|)^{\alpha+1}\,dz\,d\rho\le Cy^{\alpha+n+2}.
 \end{align*}
A parallel reasoning works for the second integral. Thus $u\in\Lambda^\alpha$.

Now we shall see the converse. That is, assume that $u\in\Lambda^\alpha$. We want to prove that
condition \eqref{Lamda} holds. As $u\in\Lambda^\alpha$ then $u\in\Lambda^{\alpha'}=C^{\alpha'/2,\alpha'}_{t,x}(\R^{n+1})$
for some $\alpha' <1$. Remember that we have the uniform convergence
$\|P_yu-u\|_{L^\infty(\R^{n+1})}\to0$, as $y\to0^+$, see \eqref{uniform}.
Moreover, by Theorem \ref{caract},
$$\|y\partial_yP_yu\|_{L^\infty(\R^{n+1})}\to0,\quad\hbox{as}~y\to0^+.$$
By using these last two properties we can readily derive the identity 
$$u(t,x) =\int_0^yy'\frac{\partial^2P_{y'}u(t,x)}{(\partial y')^2}\,dy'-y\partial_yP_yu(t,x)+P_yu(t,x).$$
To prove \eqref{Lamda} we use the identity above. Hence we are reduced to estimate
second order incremental quotients of $P_yu(t,x)$ and its first and second derivatives in $y$.
We will show the computation for the case of $g(t,x)=P_yu(t,x)$, the other cases follow the same path.
For $0< \theta,\lambda <1$, $-1< \nu < 1,$  we can write 
\begin{align*}
g&(t-\tau,x+z)+ g(t-\tau,x-z) -2g(t,x)\\
&=\Big[\nabla_xg(t-\theta\tau,x+\theta z)-\nabla_xg(t-\lambda\tau,x-\lambda z)\Big]\cdot z
-\Big[ \partial_tg(t-\theta\tau,x+\theta z)+\partial_tg(t-\lambda\tau,x+\lambda z)\Big]\tau \\
&\leq |D^2_xg(t-\nu\tau,x+\nu z)|(\theta+\lambda)|z|^2 +
|D^2_{t,x}g(t-\nu\tau,x+\nu z)|(\theta+\lambda)|z|\tau \\
&\quad+|\partial_tg(t-\theta\tau,x+\theta z)|\tau+|\partial_tg(t-\lambda\tau,x+\lambda z)|\tau.
\end{align*}
Now, by using the reasonings of Lemma \ref{lema2} we have for $y=\tau^{1/2}+|z|$,
\begin{multline*}
|g(t-\tau,x+z)+g(t-\tau,x-z)-2g(t,x)|
\le C\Big[y^{\alpha-2}|z|^2+y^{\alpha-3}|z|\tau+y^{\alpha-2}\tau\Big] \\
\le C\Big[y^{\alpha-2}(\tau^{1/2}+|z|)^2+y^{\alpha-3}(\tau^{1/2}+|z|)^3+y^{\alpha-2}(\tau^{1/2}+|z|)^2\Big]
\le (\tau^{1/2}+|z|)^\alpha.
\end{multline*}
\end{proof}

\begin{proof}[Proof of Theorem \ref{prop:2} (2)]
By choosing $\tau=0$ and $z=0$ in \eqref{Lamda} we get that
$u(\cdot,t)\in\Lambda_x^\alpha(\R^n)$ uniformly in $t$ and $u(x,\cdot)\in\Lambda_t^{\alpha/2}(\R)$
uniformly in $x$. For the converse observe that, as $\alpha/2 <1$, by 
\cite[Proposition~9,~p.~147]{Stein} we have 
\begin{align*}
|u&(t-\tau,x+z)+u(t-\tau,x-z)-2u(t,x)| \\
&\le |u(t-\tau,x+z)+u(t-\tau,x-z)-2u(t-\tau,x)|+|2u(t-\tau,x)-2u(t,x)| \\
&\le \|u(t-\tau,\cdot)\|_{\Lambda_x^\alpha(\R^n)}|z|^\alpha+
2\|u(\cdot,x)\|_{\Lambda_t^{\alpha/2}(\R)}\tau^{\alpha/2}\le C (|z|+\tau^{1/2})^\alpha.
\end{align*}
\end{proof}

\begin{proof}[Proof of Theorem \ref{prop:2} (3)] 
Assume that $2<\alpha\le4$. We have
$$\|\partial^5_yP_yu\|_{L^\infty(\R^{n+1})}\le Cy^{-5+\alpha},$$
and we get 
$$\bigg\|\frac{\partial^5P_yu}{\partial y^3\,\partial x_i^2}\bigg\|_{L^\infty(\R^{n+1})}\le Cy^{-5+\alpha}.$$
Then 
$$ \frac{\partial^4 P_yu}{\partial y^2\,\partial x_i^2}=\int_y^1\frac{\partial^5P_{y'}u}{\partial{y'}^3\,\partial x_i^2}\,dy'
+\frac{\partial^4 P_yu}{\partial y^2\partial x_i^2}\bigg|_{y=1}.$$
Hence, for small $y$ we have 
$$\bigg\|\frac{\partial^4 P_yu}{\partial y^2\,\partial x_i^2}\bigg\|_{L^\infty(\R^{n+1})}
\le C\int_y^1(y')^{-5+\alpha}\,dy'+C \le Cy^{-4+\alpha}.$$
Two more integrations give that $\partial_{x_i}^2P_yu(t,x)$ is a Cauchy sequence in 
$L^\infty(\R^{n+1})$. Then the uniform limit as $y\to0^+$ is exactly $\partial^2_{x_i}u(t,x)$.
As $\|P_y\|_{L^\infty(\R^{n+1})}\leq C$, we have that
$$P_{y'}\frac{\partial^2P_yu(t,x)}{\partial x_i^2}\rightarrow P_{y'}\frac{\partial^2  u(t,x) }{\partial x_i^2},\quad\hbox{as}~y\to0.$$
But 
$$P_{y'}\frac{\partial^2 P_yu(t,x) }{\partial x_i^2}=  \frac{\partial^2 P_yP_{y'}u(t,x) }{\partial x_i^2}.$$
Hence in the limit,
$$P_{y'}\frac{\partial^2u(t,x)}{\partial x_i^2}=\frac{\partial^2 P_{y'}u(t,x)}{\partial x_i^2}.$$
Then 
$$\Bigg\| \frac{\partial^2 P_{y'}(\frac{\partial^2u(t,x)}{\partial x_i^2})}{(\partial y')^2}\bigg\|=
\bigg\|\frac{\partial^4P_{y'}u(t,x)}{(\partial y')^2\,\partial x_i^2}\bigg\| \le(y')^{-4+\alpha}.$$
That is $\partial^2_{x_i}u(t,x)\in\Lambda^{\alpha-2}$. For the derivative with respect to $t$
we start in an analogous way,
then we get
$$\bigg\|\frac{\partial^5 P_yu}{\partial y^4 \partial t}\bigg\|_{L^\infty(\R^{n+1})}\le Cy^{-6+\alpha}.$$
We continue with the obvious changes. Details are left to the interested reader.
\end{proof}

\begin{proof}[Proof of Theorem \ref{prop:2} (4)] 
The statements in $(i)$--$(v)$ follow from the previous ones. Finally for $(vi)$,
which is the case $\alpha=2$, we proceed as in the proof of part
\textit{(1)} but starting from $\displaystyle\frac{\partial^3 P_y u(t,x)}{\partial y^3}$.
\end{proof}

\medskip

\noindent\textbf{Acknowledgements.} We would like to thank Luis A. Caffarelli for pointing out
the possibility of proving the boundary Harnack inequality for fractional caloric functions,
 {and to Emmanouil Milakis and Xavier Ros-Ot\'on for their interest in our results
and methods and for providing several useful comments and references.
We are grateful to the referees for their substantial suggestions and careful reading
that helped us to improve the presentation of the paper.}



\end{document}